\documentclass[11pt]{amsart}
\usepackage{amsmath, amsthm,  amsfonts}
\usepackage{url}

\usepackage[a4paper, margin=3cm]{geometry}
\usepackage{amssymb}
  
  \input xypic
  \xyoption{all}
  
  \usepackage{amsmath, amsthm,  amsfonts}

  \theoremstyle{plain}
  \newtheorem{theorem}{Theorem}[section]
  
  \newtheorem{proposition}[theorem]{Proposition}

  \theoremstyle{definition}
  \newtheorem{definition}[theorem]{Definition}

  \newtheorem{remark}[theorem]{Remark}
    \newtheorem{example}[theorem]{Example}

  \numberwithin{equation}{section}
  \numberwithin{figure}{section}

  \renewcommand{\cH}{{\mathcal H}}
   \newcommand{\cU}{{\mathcal U}}
  \newcommand{\cQ}{{\mathcal Q}}
  
  \newcommand{\cE}{{\mathcal E}}

  \renewcommand{\cD}{{\mathcal D}}
  \newcommand{\cP}{{\mathcal P}}
  \newcommand{\cC}{{\mathcal C}}
  \newcommand{\cG}{{\mathcal G }}
  
  \newcommand{\U}{{\mathbf U}}

   \newcommand{\ba}{\begin{eqnarray}}
   \newcommand{\na}{\end{eqnarray}}
   \newcommand{\ban}{\begin{eqnarray*}}
   \newcommand{\nan}{\end{eqnarray*}}

  \newcommand{\Fred}{{\bf  Fred }}
   
    \renewcommand{\Im}{{\bf Im }}
     
  \newcommand{\PU}{{\bf  PU}}
  
   \newcommand{\Spec}{{\bf  Spec}}
 
\newcommand{\BSO}{{\bf  BSO}}
\newcommand{\BSpin}{{\bf  BSpin}}
 \newcommand{\Td}{{\bf  Td}}

 \newcommand{\Cl}{{\bf  Cliff}}

  \newcommand{\CC}{{\mathbb C}}
  \newcommand{\RR}{{\mathbb R}}
  \newcommand{\ZZ}{{\mathbb Z}}

  \renewcommand{\a}{\alpha}
  
\renewcommand{\l}{\lambda}

    \newcommand{\disp}{\displaystyle}

\begin{document}

  \title[Differential   Twisted K-theory and  Applications]
{Differential   Twisted K-theory and Applications}

  \author[A.L. Carey]{Alan L. Carey}
  \address[Alan L. Carey]
  {Mathematical Sciences Institute\\
  Australian National University\\
  Canberra ACT 0200 \\
  Australia}
  \email{acarey@maths.anu.edu.au}
  \author[J. Mickelsson]{Jouko Mickelsson}
  \address[Jouko Mickelsson]
  {Department of Mathematics and Statistics\\
  University of Helsinki\\
  Finland\\  and Department of Theoretical Physics\\KTH \\Stockholm \\Sweden}
  \email{jouko.mickelsson@helsinki.fi}

  \author[B.L.  Wang]{Bai-Ling Wang}
  \address[Bai-Ling Wang]
  {Department of Mathematics\\
    Mathematical Sciences Institute\\
    Australian National University\\
    Canberra, ACT 0200\\ Australia.
}
  \email{wangb@maths.anu.edu.au}


  \begin{abstract} In this paper, we develop differential 
  twisted K-theory
and   define a twisted Chern character on twisted K-theory which depends on 
a choice of connection and curving on the twisting gerbe. 
We also establish the general Riemann-Roch theorem in twisted K-theory and find
some applications in the study of twisted K-theory of  compact  simple Lie groups.
   \end{abstract}
  \maketitle
\tableofcontents

  \section{Introduction}

 Generalized differential cohomology theories 
have recently excited considerable interest.
For example, Cheeger-Simons differential characters play a 
role in index theory.
  In \cite{Lott}, Lott developed $\RR/\ZZ$-valued index theory for  
Dirac operators
  coupled to virtual complex vector bundles with trivial Chern character, 
extending Atiyah-Patodi-Singer's reduced eta-invariants
  for flat vector bundles.  The resulting K-theory is called  
K-theory with $\RR/\ZZ$ coefficients. It  is related to what is  now called
`differential K-theory' as proposed
  by Freed \cite{Fre1},  Hopkins and Singer \cite{HopSin}, and further developed in  \cite{BunSch}. Our aim here is to extend some of these ideas to
  create twisted differential K-theory.
   
   Henceforth  $X$ will always denote a smooth manifold,
$\cH$ an infinite dimensional separable complex  Hilbert space and $\PU(\cH)$ the projective unitary group of $\cH$.  The term `twist' in twisted K-theory  $K^{ev/odd}(X, \sigma)$
   of $X$ refers to a continuous map
 $\sigma: X\to K(\ZZ, 3)$ which necessarily determines a principal  
 $\PU(\cH)$-bundle $\cP_\sigma$
  over $X$.  

We have been motivated to write this paper by 
some questions from physics.
The reader unfamiliar with the language can pass on to the next paragraph.
K-theory has previously been used to study anomaly
  cancellation problems
 for action principles in the presence of
D-branes in string theory and M-theory.
 For Type II superstring theory with non-trivial B-field on $X$,
 it is believed  that Ramond-Ramond charges
 lie in twisted K-theory  $K^{ev/odd}(X, \sigma)$
   of $X$ (Cf.\cite{Wit1} \cite{Wit2}) with the twist, as above, given by
 $\sigma: X\to K(\ZZ, 3)$.  
The cohomology classes  of Ramond-Ramond charges are twisted cohomology classes.

 To describe this twisted cohomology we start with
  a closed differential form $H$ representing
 the image of $[\sigma]$ under the map $H^3(X, \ZZ)\to H^3(X, \RR)$.  
 We then need some additional concepts, reviewed in Section 2, 
 beginning with the canonical  `lifting bundle gerbe' 
$\cG_\sigma$ over $X$ determined by $\cP_\sigma$. It may be equipped with a
gerbe connection $\theta$ and curving $\omega$.
Then $H$ is in fact the normalized curvature of  the triple
$
 \check{\sigma} = (\cG_\sigma, \theta, \omega). 
 $
 Next, 
 twisted cohomology $H^{ev/odd}(X, d-H)$ is the cohomology of the complex of differential forms
 on $X$ with the coboundary operator given by
 $d-H$.  It is known (Cf. \cite{AS2}, \cite{BCMMS}, \cite{MatSte})  that  the twisted Chern character 
 \[
Ch_{\check{\sigma}}: K^{ev/odd}(X, \sigma)  \longrightarrow H^{ev/odd} (X, d-H).
\]
depends on a choice of a gerbe connection  and a curving  on  the underlying gerbe $\cG_\sigma$.

 In many applications,  the  twisted Chern character  is trivial due to the fact that the  twisted
 K-groups are torsion. This means that the corresponding twisted cohomology $H^*(X, d-H)$ is zero in these examples.  In order to study  torsion elements in twisted K-theory, we 
 give a geometric construction of  a
 `differential twisted K-theory   with twisting given by   ${\check\sigma}$'.
 That is, we replace ordinary twisted K-theory by a more refined
cohomology theory that depends not only on the twist $\sigma$ 
but also on the gerbe connection and curving. 
More specifically, in Section 3,  
we   will construct  our differential twisted K-theory, denoted   $\check{K}^{ev/odd}(X, \check{\sigma})$ and a
differential twisted Chern character  $ ch_{ \check{\sigma}} $ on $\check{K}^{ev/odd} (X, \check{\sigma})$. We show that it gives   a  well-defined map   
     \ba\label{ch}
   Ch_{ \check{\sigma}}: K^{ev/odd} (X, \sigma)  \longrightarrow H^{ev/odd}(X, d-  H),
   \na
  such that the following diagram commutes:
    \ba\label{c:1}
  \xymatrix{
    \check{K}^{ev/odd}  (X, \check{\sigma}) \ar[r]\ar[d]^{ch_{\check{\sigma}}}
      & K^{ev/odd}(X, \sigma)\ar[d]^{Ch_{\check{\sigma}}}   \\
  \Omega^{ev/odd}_0(X,  d- H) \ar[r]  & H^{ev/odd}(X, d- H), }
\na
where we have used the notation $ \Omega^{ev/odd}_0(X, d- H) $ to denote the image of 
the  differential Chern character  map
\[
ch_{\check{\sigma}}: \check{K}^{ev/odd}(X, \check{\sigma}) \longrightarrow
 \Omega^{ev/odd}(X,  d- H ).
\]
In fact, the commutative diagram (\ref{c:1})  is part of a  commutative interlocking  diagram of exact sequences  which characterizes   differential twisted K-theory   $\check{K}^{ev/odd}(X, \check{\sigma})$, see      Theorems  \ref{diff:twsitedK1} and \ref{diff:twsitedK0} and Remark \ref{locked:terms}.   

We  remark that in the 
construction we need to choose
local trivializations of the principal $\PU(\cH)$-bundle $\cP_\sigma$. In addition we also need the notion 
of a `spectral cut' (defined later on)  for a $\PU(\cH)$-equivariant  family of self-adjoint operators.  Proposition \ref{equivalence} ensures that 
differential twisted K-theory and its differential Chern character don't depend on these choices.
We also give a universal model of differential twisted K-theory  (cf. Theorems  \ref{universal:0} and \ref{universal:1})   such  that  differential Chern character
can be defined directly by a choice of  a  connection on a certain universal bundle classifying twisted
K-theory. 

When the twisting $\sigma: X \to K(\ZZ, 3)$ is trivial our construction gives rise to a different model for the
ordinary differential K-theory previously studied in  \cite{BunSch}, \cite{Fre1} and \cite{HopSin}.
Roughly stated the main idea is to provide some
`geometric cycles' as models for differential K-theory.
 Hopkins-Singer (\cite{Fre1},  \cite{HopSin})
apply homotopy theory to  the space of differential functions on the products of $X$ with a varying simplex, taking values in a classifying space for K-theory.  The viewpoint of \cite {BunSch}
is to apply the  family index theorem, so their   models
are geometric fibrations equipped with  
    elliptic differential operators. 
 We have utilized the spectral theory of self-adjoint operators to get local bundles over a groupoid, the 
double intersections of an open cover, for the odd case of twisted differential K-theory (and an adaptation of this for the even case).
    
  In Section 4  we establish a 
 generalization of   the Atiyah-Hirzebruch 
  Riemann-Roch theorem in  twisted K-theory. Note that 
   the Riemann-Roch theorem in twisted K-theory for $K$-oriented maps was discussed in   \cite{BEM}.
   Our refinement in Section 4 is to extend the theorem
 (see Theorems \ref{RR:1} and Theorem \ref{RR:2})  so that it applies to non-$K$-oriented
 smooth maps $f: X\to Y$.  When $[\sigma] =0$ and $f$ is $K$-oriented,  our theorem   reduces to   the standard Atiyah-Hirzebruch
  Riemann-Roch theorem (\cite{AtiHir}) .

In order to apply these ideas to string geometry we need one further concept, namely,  the
notion of twisted  eta forms for those twisted K-classes 
that are torsion. This is useful in studying non-equivariant
twisted K-theory  of a compact Lie group.  We discover that these twisted  eta forms
  can be used to distinguish different twisted K-classes. Explicit computations are done
  for the Lie groups $SU(2)$ and $SU(3)$ as these examples arise in
  conformal field theory.

\noindent {\bf Acknowledgements}. We 
thank the referee for carefully reading the manuscript and providing helpful advice.
We also thank Thomas Schick for his comments on  the manuscript.
We acknowledge the assistance of Michael Murray for the discussion of Deligne cohomology.
The authors acknowledge  the support of: the Erwin Schr\"odinger Institute (AC, JM, BW),
 the Australian
  Research Council (AC, BW),  the Clay Mathematics Institute (AC) and
   the Academy of Finland grant 516 8/07-08 (JM).

\section{Preliminary notions}

  \subsection{Twisted K-theory }  

In this subsection, we briefly review some basic facts about twisted K-theory, the main references 
 are \cite{AS1}, \cite{BCMMS}, \cite{CW2}
 and \cite{Ros}. Let  $\PU(\cH)$ denote  the projective unitary group (equipped with the norm topology)  of an infinite dimensional separable complex  Hilbert space $\cH$
and recall (cf. \cite{Kui})  that it has the homotopy type of an 
Eilenberg-MacLane space $K(\ZZ, 2)$. 
The  classifying space   $B\PU(\cH)$ for  principal $\PU(\cH)$-bundles   is an 
Eilenberg-MacLane space $K(\ZZ, 3)$. 
Thus, the set of isomorphism classes of 
principal $\PU(\cH)$-bundles over   $X$ is canonically identified
with the space $[X, K(\ZZ, 3)]$ of homotopy classes of maps from $X$ to
$K(\ZZ, 3)$ which in turn (Proposition 2.1 in \cite{AS1}) is just
$H^3(X, \ZZ).$

Much of the theory applies to  locally compact,
metrizable and separable
   spaces $X$, however, we will remain with the case where
   $X$ is a manifold. Let  $\cP_\sigma$  be a principal
$\PU(\cH)$-bundle  over $X$ whose classifying map $\sigma: X \to K(\ZZ, 3)$, called a twisting,  defines
 an element  $[\sigma]  \in H^3(X, \ZZ).$
   There is a well-defined conjugation action
   of $\PU(\cH)$ on the space $\Fred(\cH)$ of bounded Fredholm operators.
   Let
   $\Fred(\cP_\sigma)= \cP_\sigma \times_{\PU(\cH)} \Fred(\cH)  $
   be the associated bundle of Fredholm operators.

 \begin{definition}  If $X$ is compact,  the twisted K-group $K^0 (X, \sigma)$ is
  defined to be the space of homotopy classes of sections of $\Fred(\cP_\sigma)$
(see \cite{Ros}).
   If $X$ is locally compact, then the twisted K-group
$K^0 (X, \sigma)$ is defined to be the space of homotopy classes of
`compactly supported sections'
of $\Fred (\cP_\sigma)$, where compactly supported means
these sections take values in the invertible operators away from a compact set.
 If $X_0$ is a closed subset of $X$, we may define
in the obvious way the relative twisted K-group
 $$K^0 (X, X_0; \sigma) = K^0 (X-X_0, \sigma).$$
Replacing $\Fred (\cH)$ by the space $\Fred^{sa}_*$ of self-adjoint Fredholm operators with both
 positive and negative essential spectrum, we can
similarly define the twisted K-group $K^1(X, \sigma)$.
\end{definition} 

\begin{remark}   In the applications we consider unbounded 
self adjoint Fredholm operators of the form $D+A$ where $D$
is a fixed  unbounded self adjoint operator and $A$ is bounded and varies. We utilise the topology on the unbounded self adjoint Fredholm operators
induced by the map $D+A\to F_{D+A}=(D+A)(1+(D+A)^2)^{-1/2}$ \cite{CP1}.
This is justified for example by Theorem 8 in Appendix A of \cite{CP1} which proves the estimate
$||F_{D+A}-F_D||<||A||$. This estimate implies that if $A$ varies smoothly in the uniform norm then so does $F_{D+A}$ and hence in this sense so does $D+A$. 
\end{remark}

  Next we  choose
   a local trivialization  of  $\cP_\sigma$ with respect  to a good open cover $X = \bigcup_i U_i$ with
   transition functions  given by
$g_{ij}: U_i \cap U_j  \longrightarrow \PU(\cH)$ and lifts
$\hat{g}_{ij}:  U_i \cap U_j  \longrightarrow \U(\cH),$ for each $g_{ij}$.
Then we have
\ba\label{{C}ech}
\sigma_{ijk} \cdot Id  = \hat{g}_{ij}\hat{g}_{jk}\hat{g}_{ki}
\na
for a $U(1)$-valued \v{C}ech cocycle $\sigma_{ijk}: U_{ijk} \to U(1)$, where
$U_{ijk} = U_i\cap U_j \cap U_k$ and $Id$ is the identity operator in $\cH$.

Then an element in $K^0 (X, \sigma)$ can be represented by either a continuous
$\PU(\cH)$-equivariant map $f: \cP_\sigma \to \Fred(\cH)$, or by a twisted family of locally defined functions
$\{
f_i: U_i \longrightarrow \Fred ( \cH )
\}$
satisfying $f_i= \hat{g}^{-1}_{ij} f_j  \hat{g}_{ij}$.    Similarly,
an  element in $K^1 (X, \sigma)$ can be represented either by a continuous
$\PU(\cH)$-equivariant map $f: \cP_\sigma \to \Fred_*^{sa}$, or  a twisted family of locally defined functions
$\{
f_i: U_i \longrightarrow \Fred^{sa}_* 
\},$
satisfying $f_j = \hat{g}^{-1}_{ij} f_i \hat{g}_{ij}.$

Let $\sigma_1,    \sigma_2:  X\to K(\ZZ, 3)$ be a pair of continuous maps.  We need to define a map 
  $ \sigma_1   +  \sigma_2  $ from $X$ to $K(\ZZ, 3)$  such that 
  \[
[\sigma_1   +  \sigma_2 ] = [\sigma_1]+[ \sigma_2 ]
\]
in $H^3(X, \ZZ)$.
 In order to define this map, we fix an isomorphism $\cH \otimes \cH \cong \cH$
 which induces a group homomorphism
 $
 \U(\cH) \times \U(\cH) \longrightarrow \U(\cH)
 $
 whose restriction to the center is the group multiplication on $U(1)$. So we have a 
 group homomorphism
 \[
 \PU(\cH) \times \PU(\cH) \longrightarrow \PU(\cH)
 \]
 which defines  a continuous map, denoted $m_\ast$, of CW-complexes
 \[
 B \PU(\cH) \times B \PU(\cH)  \longrightarrow  B\PU(\cH).
 \]
 As $B\PU(\cH)$ is a $K(\ZZ, 3)$,  we may think of this as  a continuous map taking 
 $
 K(\ZZ, 3) \times K(\ZZ, 3)$ to $K(\ZZ, 3), 
 $
 which can be used to define 
 \[
\xymatrix{
\sigma_1   +  \sigma_2 :  \quad   X\ar[r]^{(\sigma_1,  \sigma_2  )  \ \  } & K(\ZZ, 3) \times K(\ZZ, 3) \ar[r]^{
\qquad m_*}   & 
K(\ZZ, 3). }
\]

  Just as in ordinary K-theory, we can define twisted K-groups $K^i(X, \sigma)$ for all
  $i\in \ZZ$ such that the twisted K-theory is a generalized cohomology theory with period $2$
  on the category of topological spaces equipped with a principal $\PU(\cH)$-bundle.
     Twisted K-theory satisfies the following basic properties  \cite{AS1}\cite{CW2}.\\
(1) For any proper continuous map $f: X\to Y$
  there exists a natural pull-back map
 \ba\label{pull:back}
 f^*: K^i(Y, \sigma) \longrightarrow  K^i (X, f^*\sigma),
 \na
for any $\sigma: Y \to K(\ZZ, 3)$.\\
(2) With a fixed isomorphism $\cH\otimes \cH \to \cH$, there is a  homomorphism
\[
K^i(X, \sigma_1) \times K^j(X, \sigma_2)  \longrightarrow 
K^{i+j} (X, \sigma_1 +\sigma_2)
\]
for any two twistings $\sigma_1,    \sigma_2:  X\to K(\ZZ, 3)$.   \\
(3)
  If $X$ is covered by two closed  subsets $X_1$ and $X_2$,
    there is a Mayer-Vietoris  exact sequence
 \[
 \xymatrix{
K^1 (X_1,\sigma_1 )\oplus K^1( X_2,\sigma_2 )\ar[r]& K^1  (X_1\cap X_2, \sigma_{12})\ar[r]&   K^0 (X, \sigma)
  \ar[d]\\
K^1 (X, \sigma) \ar[u] & K^0  (X_1\cap X_2,\sigma_{12} ) \ar[l]&
 K^0  (X_1, \sigma_1) \oplus K^0(X_2,\sigma_2 )     \ar[l]
   }
 \]
   where $\sigma_1$, $\sigma_2$ and $\sigma_{12}$ are the restrictions of $\sigma: X\to K( \ZZ, 3)$  to
   $X_1$, $X_2$ and $X_1\cap X_2$ respectively.\\
(4)
  If $X$ is covered by two  open subsets $U_1$ and $U_2$,
    there is a Mayer-Vietoris  exact sequence
 \[
 \xymatrix{
 K^0 (X, \sigma)\ar[r]& K^1  (U_1\cap U_2, \sigma_{12})\ar[r]& K^1 (U_1,\sigma_1 )\oplus K^1( U_2,\sigma_2 )
  \ar[d]\\
 K^0  (U_1, \sigma_1) \oplus K^0(U_2,\sigma_2 )\ar[u] & K^0  (U_1\cap U_2,\sigma_{12} ) \ar[l]&
   K^1 (X, \sigma) \ar[l]
   }
 \]
   where $\sigma_1$, $\sigma_2$ and $\sigma_{12}$ are the restrictions of $\sigma: X\to K( \ZZ, 3)$ to
   $U_1$, $U_2$ and $U_1\cap U_2$ respectively.\\
(5) ({\bf Thom isomorphism})  Let $\pi: E\to X$ be an oriented real vector bundle of rank $k$  over $X$ with the 
 classifying map
  denoted by $\nu_E: X \to \BSO (k)$, then there is a canonical isomorphism, for any twisting $\a: X\to K(\ZZ, 3)$,
 \ba\label{Thom:CW}
 K^{ev/odd}(X, \a + W_3 \circ \nu_E  )  \cong K^{ev/odd} (E, \a \circ  \pi),
 \na
with the grading shifted by $k (mod \ 2 )$. 
 Here  $W_3: \BSO (k) \to  K(\ZZ, 3)$ is   the classifying map of the principal 
$K(\ZZ, 2)$-bundle $\BSpin^c (k) \to \BSO(k) $. \\
(6) For any differentiable map $f: X\to Y$ between two smooth manifolds $X$ and $Y$, there
  is a natural push-forward map
  \ba\label{push:forward}
  f^K_!:   K^i \bigl(X, f^*\sigma +   W_3\circ \nu_f \bigr) \longrightarrow K^{i+ d(f)}(Y, \sigma),
  \na
  for any $\sigma\in H^3(Y, \ZZ)$, $d(f)= \dim (X)-\dim (Y) \mod  2 $, and
  $ \nu_f : X \to \BSO$ is a classifying map of $TX \oplus f^*TY$.  Note that
   $W_3\circ \nu_f $ is a map $X\to K(\ZZ, 3) $ representing   the image of $w_2(X) +  f^*(w_2(Y)) \in H^2(X, \ZZ_2)$
   under the Bockstein homomorphism
 $
 H^2(X, \ZZ_2) \longrightarrow H^3(X, \ZZ).
  $

 \subsection{Geometry of bundle gerbes}

 In this subsection, we recall some basics of bundle gerbes and their connections and curvings from \cite{Mur,MurSte}. First recall that associated to the central extension
\ba\label{cen:ext}
1\to U(1) \longrightarrow U(\cH) \longrightarrow \PU(\cH) \to 1.
\na
there exists a so-called lifting bundle gerbe
 $\cG_\sigma$ \cite{Mur}. This object is represented diagrammatically by
\ba\label{bg:sigma}
\xymatrix{
 \cG_\sigma \ar[d]& \\
 \cP_\sigma  ^{[2]}\ar@< 2pt>[r]^{\pi_1} \ar@< -2pt>[r]_{\pi_2}
  & \cP_\sigma    \ar[d]^{\pi} \\
   &X}.
\na
This diagram should be read as follows.
Starting with $\pi:  \cP_\sigma\to X$ form the fibre product $\cP_\sigma  ^{[2]}$
which is a  groupoid
$\cP_\sigma^{[2]} = \cP_\sigma \times_X \cP_\sigma$ with 
source and range maps $\pi_1: (y_1, y_2) \mapsto y_1$ and $\pi_2: (y_1, y_2)\mapsto y_2$. There is an obvious map from each fiber of $\cP_\sigma^{[2]}$ to $\PU(\cH)$ and so we can define
the fiber of  $\cG_\sigma$ over a point in $\cP_\sigma^{[2]}$ by pulling back the 
fibration (\ref{cen:ext}) using this map. 
This endows $\cG_\sigma$ with  a groupoid 
structure (from the multiplication in $U(\cH)$) and in fact it is a $U(1)$-groupoid extension of $\cP_\sigma  ^{[2]}$.

We will review
the notion of a bundle gerbe connection, curving and
curvature on $\cG_\sigma$.   Recall that there is an exact sequence of
differential $p$-forms
\ba\label{exact}
\xymatrix{
\Omega^p(X) \ar[r]^{\pi^*} & \Omega^p(\cP_\sigma) \ar[r]^\delta &\Omega^p(\cP^{[2]}_\sigma) \ar[r]^\delta & \cdots, }
\na
Here $\cP_\sigma^{[q]}$ is the $q$-fold fiber product of $\pi$,   $\delta=\sum_{i=1}^{q+1}(-1)^i \pi^*_i:
\Omega^p(\cP_\sigma^{[q]}) \to
\Omega^p(\cP_\sigma^{[q+1]}) $  is the alternating sum of pull-backs 
 of projections,
where $\pi_i$  is the projection map which omits the $i$-th point in the fibre product.

 A gerbe  connection on $\cP_\sigma$ is a unitary connection
$\theta$ on the principal $U(1)$-bundle $\cG_\sigma$ over 
$\cP_\sigma$ which  commutes with the bundle gerbe product.    
A bundle gerbe connection $\theta$ has curvature
\[
F_\theta \in \Omega^2(\cP_\sigma)
\]
 satisfying $\delta (F_\theta) = 0$
  and hence from the
exact sequence (\ref{exact}) for $p=2$,  there exists a two-form $\omega$  on  $\cP_{\sigma}$ such that
\[
F_\theta = \pi_2^* (\omega) -\pi_1^* (\omega).
\]
Such an $\omega$  is called a curving for the gerbe connection $\theta$. The choice of a curving is not unique,
  the ambiguity in the choice is precisely the addition of the pull-back 
  to $\cP_{\sigma}$ of
a two-form on  $X$.   Given a choice of curving $\omega$ we  have
\[
\delta (d\omega ) = d\delta (\omega ) = d F_\theta  = 0
\]
so that we can find   a unique closed  three-form on $\beta $ on $X$, such that $d\omega  = \pi^*  \beta $.    We  denote by 
$\check{\sigma} = (\cG_\sigma, \theta, \omega)$
 the lifting bundle  gerbe $\cG_\sigma$ with the connection $\theta$ and  a curving $\omega$. 
 Moreover
 $H= \dfrac{\beta}{2\pi \sqrt{-1}}$  is a de Rham representative for
the Dixmier-Douady class $[\sigma]$. We
 call  $H$  the normalized curvature of  $\check{\sigma}$.

Consider a good open cover $\cU = \{U_i\}$ of $X$
such that $\cP_\sigma \to X$ has  trivializing sections $\phi_i$ over each $U_i$ with transition functions
$g_{ij}:  U_i \cap U_j \longrightarrow \PU(\cH)$
satisfying  $\phi_j = \phi_i g_{ij}$ over $U_i\cap U_j$. 
Define $\sigma_{ijk}$ by   $\hat{g}_{ij} \hat{g}_{jk}= \hat{g}_{ik}\sigma_{ijk} $ for 
 a lift of $g_{ij}$ to $\hat{g}_{ij}:  U_i \cap U_j \to \U (\cH) $.
 Note that the pair $(\phi_i, \phi_j)$  defines a section of $\cP^{[2]}_{\sigma}$
over $U_i \cap U_j$. 
  The connection $\theta$ can be pulled back by $(\phi_i, \phi_j)$
to define  a 1-form
$A_{ij}$ on $U_i\cap  U_j$ and the curving $\omega$ can be pulled-back by the $\phi_i$
 to define two-forms $B_i$ on $U_i$. Then
the triple
\ba\label{deligne:cocycle}
(\sigma_{ijk}, A_{ij}, B_i)
\na
 is a   degree two smooth Deligne cocycle. 
 
 In our construction of the Chern character for differential twisted K-theory and in our proof of the Riemann-Roch theorem  for twisted K-theory we need to work with the Deligne data defined by  $\check{\sigma}$. So we 
  recall for the reader's convenience  that  degree $p$ smooth Deligne cohomology
 is the $p$-th  \v{C}ech  hypercohomology group of the complex of sheaves on $X$ (Cf. \cite{Bry}):
  \[
  \underline{U(1)}  \stackrel{d \text{log}} {\longrightarrow} \Omega^1_X
    \stackrel{d} {\longrightarrow}  \Omega^2_X \longrightarrow \cdots 
    \stackrel{d} {\longrightarrow} \Omega^p_X
  \]
  where $\underline{U(1)}$ is the sheaf of germs of smooth $U(1)$-valued functions
on $M$ and
  $\Omega^p_X$  is the sheaf of germs of  imaginary-valued differential $p$-forms on $X$.   
 A degree 0  smooth   Deligne class is  represented by a smooth map $f \colon X \to U(1)$.
A degree 1  smooth Deligne class $\xi$ can be represented by  
  Hermitian line bundles with Hermitian  connection.   

 The degree 2 Deligne cohomology group $H^2_\cD(X)$
  can be calculated as the cohomology of the total complex of the double complex with respect to
  a good cover $\cU=\{U_i\}$  of $X$:
  \ba\label{Del}
  \xymatrix{
  \vdots & \vdots & \vdots   \\
  C^2(\cU, \underline{U(1)} ) \ar[r]^{d\log} \ar[u]^{\delta} &  C^2(\cU,  \Omega^1_X) 
  \ar[r]^{d } \ar[u]^{\delta} &  C^2(\cU,  \Omega^2_X)    \ar[u]^{\delta}  \\
   C^1(\cU, \underline{U(1)} ) \ar[r]^{d\log} \ar[u]^{\delta} &  C^1(\cU,  \Omega^1_X) 
  \ar[r]^{d } \ar[u]^{\delta} &  C^1(\cU,  \Omega^2_X)  \ar[u]^{\delta}    \\
 C^0(\cU, \underline{U(1)} ) \ar[r]^{d\log} \ar[u]^{\delta} &  C^0(\cU,  \Omega^1_X) 
  \ar[r]^{d } \ar[u]^{\delta} &  C^0(\cU,  \Omega^2_X)    \ar[u]^{\delta}  }   
    \na
which is the quotient of  the abelian group of degree two Deligne cocycles by  the subgroup of degree two Deligne coboundaries.  Here 
a triple 
\[
(\sigma_{ijk}, A_{ij}, B_i) 
\in   C^2(\cU, \underline{U(1)} ) \oplus C^1(\cU,  \Omega^1_X) \oplus  C^0(\cU,  \Omega^2_X) 
\]
   is a  Deligne cocycle if it satisfies the following cocycle  condition
\begin{enumerate}
\item $ \sigma_{ijk}\sigma_{ijl}^{-1} \sigma_{ikl} \sigma_{jkl}^{-1} =1.$
\item $A_{ij} + A_{jk}  + A_{ki} = d \log (\sigma_{ijk})$.\item   $B_j - B_i=  d A_{ij}.$
\end{enumerate} 
 A degree two Deligne coboundary is a triple of the following type
 \[
 (h_{ij} h^{-1}_{ik} h_{jk}, d\log (h_{ij}) +a_i -a_j, da_i) 
 \]
for $(h_{ij}, a_i)  \in C^1(\cU, \underline{U(1)} ) \oplus  C^0(\cU,  \Omega^1_X) $. 
The degree 2  Deligne cohomology of $X$ classifies   stable isomorphism classes 
  of   Hermitian bundle gerbes  with connection and curving (see \cite{MurSte}).

We conclude this subsection by noting 
an important consequence of the above discussion 
 namely that the degree two Deligne class of  (\ref{deligne:cocycle})   determined
by  $\check{\sigma} =\{ \cG_\sigma, \theta, \omega)$
 is independent of the choice of locally trivializing sections  of $\cP_\sigma$ and the lifting $\hat{g}_{ij}$.

\section{Differential twisted K-theory}

Recall our data:  $X$ is a compact smooth  manifold, $\cP_\sigma$ is  a  principal
$\PU(\cH)$ bundle over $X$ whose  classifying map is given by 
$\sigma: X\longrightarrow K(\ZZ, 3)$.  There is a lifting bundle gerbe $\cG_\sigma$ with a bundle gerbe
  connection $\theta$  and a  curving $\omega$ with $  \check{\sigma} = (\cG_\sigma, \theta, \omega) $ denoting this triple.
We refer to $ \check{\sigma}$  as a twisting in differential twisted K-theory.
The normalized  curvature 
   of  the bundle gerbe with connection and curving $  \check{\sigma}$ is denoted by $H$. With respect to an  open  cover $\{U_i\}$ of $X$ which
   trivializes $\cP_\sigma$ using local sections $\{\phi_i\}$, 
  we have an associated 
Deligne  2-cocycle  $ (\sigma_{ijk}, A_{ij}, B_i). $

We will  establish in the next  subsection the existence of a differential twisted K-theory and  construct   its differential twisted Chern character which induces  the  twisted Chern character
$ch_{\check{\sigma}}  : K^*(X, \check{\sigma})   \longrightarrow    H^*(X, d-H).
$

\subsection{The geometric model of  differential twisted K-theory}

 We begin with odd twisted K-theory.  Let $f : \cP_\sigma \to \Fred_*^{sa}$ be a $\PU(\cH)$
equivariant map to the bounded self-adjoint Fredholm operators with both
positive and negative essential spectrum, i.e., the homotopy
class of $f$ is an element of $K^1(X,\sigma).$ We make the additional
assumption that the operators $f (x)$ have discrete spectrum (which is the case in
many physics examples). This is no real limitation, since by \cite{ASin}  the space $\Fred_*^{sa}$ 
is homotopy equivalent to the subspace ${\mathcal F}_*^{sa}$ of operators of norm less than or equal to one and
with essential spectrum $\pm 1.$ Moreover the map from unbounded to bounded Fredholms
$D\to D(1+D^2)^{-1/2}$ introduced earlier maps the unbounded self adjoint operators with discrete spectrum into the space ${\mathcal F}_*^{sa}$. 

Next choose an open cover $\{U_i\},$  $i=1,2,\dots, n,$ of $X$ such that on each $U_i$  there is a local section 
$\phi_i: U_i \to \cP_\sigma$ and for each $i$ there is a real number $\lambda_i$ not in the spectrum of the operators
$f(\phi_i(x))$, for all $x\in U_i$.   We refer to  $\lambda_i$ as a choice of {\it spectral cut} for the family
  $\{ f(\phi_i(x))\}_{x\in U_i}.$ Furthermore, we can require that in each interval $(\lambda_i, \lambda_j)$ there are
only finite number of eigenvalues and each with finite multiplicity using the model ${\mathcal F}^{sa}_*$
above and selecting the $\lambda_i$'s in the open interval $(-1,1).$ 
Then  over each $U_{ij}=U_i\cap U_j$ we have a finite rank vector bundle $E_{ij}$ spanned
by the eigenvectors of $f(\phi_j(x))$ with eigenvalues in the open interval
$(\lambda_i, \lambda_j)$, with a chosen ordering $\lambda_i < \lambda_j$ for $i<j.$

Let the  transition functions be denoted by 
$g_{ij} : U_{ij} \to \PU(\cH)$  
such that  $\phi_j(x)= \phi_i(x) g_{ij}(x)$.
If $\hat{g}_{ij}$ is a lift of $g_{ij}$
to $\U(\cH)$ then
$\hat{g}_{ij} \hat{g}_{jk} \hat{g}_{ki} =\sigma_{ijk}  Id$
on triple intersections  $U_{ijk} = U_i \cap U_j\cap U_k$ with $\sigma_{ijk}$ taking values in $U(1).$

Next we define bundle maps $\phi_{ijk}: E_{ij}\to E_{ik}$   
over $U_{ijk} $ for $\lambda_i<\lambda_j<\lambda_k$ as follows. Note that  $f(\phi_j(x)) = \hat{g}_{kj}^{-1}  f(\phi_k(x))   \hat{g}_{kj}$.
First act by $\hat{g}_{kj}$ on $E_{ij}$ and then use the inclusion to $E_{ik}.$  
The vector bundle $\phi_{ijk}(E_{ij})$ can be 
identified as the tensor product $L_{kj} \otimes E_{ij}$ where the complex line bundle $L_{kj}$ over 
$U_{kj}$ comes from the lifting bundle gerbe, i.e., from the pull-back of the central extension of $\PU(\cH)$ with respect 
to the map $g_{kj}: U_{jk} \to \PU(\cH).$  Thus we have
\ba\label{twisted:bundle}
 L_{kj} \otimes E_{ij} \oplus E_{jk} = E_{ik}, 
 \na
 over $U_{ijk}$. We call this the twisted cocycle property of families of local
 vector bundles $\{E_{ij}\}$.

In the untwisted case we can identify $E_{ji}$ as the virtual bundle $-E_{ij}$.  In the twisted case 
we have to remember that the vector bundles are defined using the local sections attached to the 
second index. Therefore we identify 
 $E_{ji} $ with $- L_{ji} \otimes E_{ij}$ 
the twist coming again from the transition function $g_{ji}$ relating the local sections on open 
sets $U_i, U_j.$ 
 
Since the vector bundles $E_{ij}$ are defined via projections
$P_{ij}$  onto finite dimensional subspaces  $P_{ij} \cH$, they come equipped with a natural
connection $\nabla_{ij}$ and we can extend the above equality (\ref{twisted:bundle}) to
\ba
\label{twisted:connection}
 (L_{kj}, A_{kj}) \otimes (E_{ij}, \nabla_{ij}) \oplus (E_{jk}, \nabla_{jk}) = (E_{ik}, \nabla_{ik})
\na
where $ A_{ij}$ is the gerbe connection  determined by  the differential twisting
$\check{\sigma} = (\cG_\sigma, \theta, \omega)$.   For simplicity,
we normalize $B_i$ such that the first Chern class $c_{ij}$ of  $L_{ij}$ is represented
by $B_j-B_i$.

\begin{remark} In the case of a trivial $\PU(\cH)$ bundle one has
a choice of lifts such that  $\sigma_{ijk}=1$
and actually one has a global family of Fredholm operators parametrized
by points on $U_{ij}$.  In the untwisted  case,  the spectral subspaces $E_{ij}$ are directly
parametrized by points in $X$  and we have $E_{ij}\oplus E_{jk}= E_{ik} $ over $U_{ijk}$.
\end{remark}

\begin{remark} There is an inverse map from the twisted cocycle of local vector bundles to a global 
 object in $K^1(X, \sigma).$ To make the construction simple, we use the alternative classifying 
 space $\U_1(\cH)$ of unitary operators which differ from the unit by a trace-class operator.
 Any vector bundle $E_{ij}$ can be defined  using a projection valued map $P_{ij}: U_{ij} \to L(\cH).$ 
 We  assume that the projections $P_{ij}(x)$  commute if the second indices are equal. (This is 
 automatically the case when the projections are defined from a twisted K-theory element as
 above.) 
  On the overlap $U_{ijk}$ we require
\ba\label{coc}
 \hat g_{jk}^{-1}  P_{ij} \hat g_{jk} + P_{jk} = P_{ik}, 
 \na
this being the twisted cocycle property (\ref{twisted:bundle})  of the vector bundles $\{E_{ij}\}.$ \
 We put 
 $$ g_i(x) = e^{2\pi i \sum_j \rho_j P_{ji}((x)}$$ 
 where $\sum \rho_i(x) =1$ is a partition of unity subordinate to the open cover $\{U_i\}.$ The function 
$g_i : U_i \longrightarrow \U_1(\cH),$ satisfies   
 $$  \hat g^{-1}_{ji} (\log\, g_j)  \hat g_{ji} =  2\pi i\, \hat g^{-1}_{ji} \sum_{k} \rho_{k} P_{kj} \hat g_{ji}, 
 = \log( g_i) - 2\pi i P_{ji}  $$ 
 by (\ref{coc}), 
 which implies 
 $$ \hat g^{-1}_{ji} g_j \hat g_{ji} = g_i e^{-2\pi i P_{ji} } = g_i$$ 
 on the overlap $U_{ij}.$  Hence, $\{g_i\}$ defines an element in  $K^1(X, \sigma).$
\end{remark}

Let $\ell_{ij}$ be the top exterior power of $E_{ij}$ and $n_{ij}$ the rank of $E_{ij}.$ Then the collection of 
integers $\{n_{ij}\}$ is an integral \v{C}ech cocycle. Furthermore, 
\ba\label{twistgerbe} L_{kj}^{n_{ij} }\otimes \ell_{ij} \otimes \ell_{jk} = \ell_{ik}.\na 
As noted in   \cite{MicPel}  in the case when  $n_{ij}$ is trivial, i.e., $n_{ij} = n_i -n_j$ for some 
locally constant integer valued functions  $n_i$,  one can define $\ell'_{ij} = \ell_{ij} \otimes L_{ij}^{n_j}$ and
one has 
$$ \ell'_{ij} \otimes \ell'_{jk} = \ell'_{ik}.$$ 
This gerbe is however defined only modulo integer powers of $L$ since one can always   shift 
$n_i \mapsto n_i + n$ for a constant $n.$ 

In any  case, we obtain from (\ref{twisted:connection})
  the cocycle relation for the Chern character forms of the vector bundles involved,
\ba
\label{cocycle:form}
 e^{B_j-B_k} \omega_{ij} + \omega_{jk} = \omega_{ik},
\na
where $\omega_{ij} = ch(E_{ij}, \nabla_{ij}) $ is the Chern character form of $(E_{ij}, \nabla_{ij})$.
  
Denote by  $\bar{\omega}_{ij} $ the even differential form $
e^{B_j}\omega_{ij}.$ We have then
\[
\delta(\bar\omega_{ij}) = \bar\omega_{jk} -\bar \omega_{ik} + \bar\omega_{ij} =0,
\qquad (d-H)\bar\omega_{ij} =0.
\]

Applying the tic-tac-toe  argument  to the following twisted \v{C}ech-de Rham double complex,
\ba\label{tic-tac}
\xymatrix{&&&  \\
     \Omega^{even}(\{U_{ijk}\})\ar[u]^{\delta}\ar[r]^{d-H} &
     \Omega^{odd}(\{U_{ijk}\})\ar[u]^{\delta}\ar[r]^{d-H} & \Omega^{even}(\{U_{ijk}\})\ar[u]^{\delta}\ar[r]^{\qquad d-H} &  \\
    \Omega^{even}(\{U_{ij}\})\ar[u]^{\delta}\ar[r]^{d-H} &
     \Omega^{odd}(\{U_{ij}\})\ar[u]^{\delta}\ar[r]^{d-H} & \Omega^{even}(\{U_{ij}\})\ar[u]^{\delta}\ar[r]^{\qquad d-H} & \\   
      \Omega^{even}(\{U_{i}\})\ar[u]^{\delta}\ar[r]^{d-H} &
     \Omega^{odd}(\{U_{i}\})\ar[u]^{\delta}\ar[r]^{d-H} & \Omega^{even}(\{U_{i}\})\ar[u]^{\delta}\ar[r]^{\qquad d-H} &  \\ 
      \Omega^{even}(X)\ar[u]^{\delta}\ar[r]^{d-H} &
     \Omega^{odd}(X)\ar[u]^{\delta}\ar[r]^{d-H} & \Omega^{even}(X)\ar[u]^{\delta}\ar[r]^{\qquad d-H} &      }
  \na
we obtain an odd-degree  $(d-H)$-closed differential form $\Theta$ via the following diagram chase:
\ba\label{Theta}
 \xymatrix{
     0 &
      &   &  \\
    \bar\omega_{ij} \ar[u]^{\delta}\ar[r]^{d-H } &
       0  &  & \\   
     \eta_i \ar[u]^{\delta}\ar[r]^{ }&
    (d-H)\eta_i \ar[u]^{\delta}\ar[r]^{  }&  0  &  \\ 
      &
    \Theta\ar[u]^{\delta}\ar[r]^{d-H} & 0\ar[u]^{\delta}&      }
  \na
  with respect to  an  open  cover $\{U_i\}. $  Specifically, we can choose
  \[
  \eta_i =\sum_k \rho_k   \bar\omega_{ki} 
  \]
satisfying  $\delta(\eta_i) =    \bar\omega_{ij}$  by direct calculation. Applying $d-H$ to $  \eta_i$
and diagram chasing, we have a locally defined odd differential form
$(d-H)\eta_i$ on each $U_i$ such that 
\[
(d-H)\eta_i  = (d-H)\eta_j
\] over $U_{ij}$, hence we obtain
a globally defined  odd differential form $\Theta$  such that
\[
\Theta |_{U_i} =  (d-H)\eta_i,
\qquad (d-H) \Theta =0.
\]
This globally defined  odd differential form $\Theta$ is     denoted by 
\[
ch_{\check\sigma} (f, \{\lambda_i\}, \{\rho_i\} ),
\]
and is   called  the twisted differential  Chern character of $f$.
Clearly it can depend on the choice of spectral cut $\{\lambda_i\}$ of $\{f\circ \phi_i \}$ associated to  local
trivializing sections $\{ \phi_i: U_i \to \cP_\sigma\}$, and on a choice of a partition of unity $\{\rho_i\}$ on $X$ subordinate  to $\{U_i\}$.

The next proposition explains how  $ch_{\check\sigma} (f, \{\lambda_i\}, \{\rho_i\} )$ 
changes if we change the choice of spectral cuts $\{\lambda_i\}$,  or the partition of unity $ \{\rho_i\}$ on $X$.  We can also see the dependence on  $\PU(\cH)$-equivariant homotopies.

\begin{proposition}  \label{equivalence} \begin{enumerate}
\item  Choose local trivializing sections  $\{ \phi_i: U_i \to \cP_\sigma\}$ and let $\{\lambda_i\}$ be a family of 
spectral cuts  for  $\{f\circ \phi_i \}$.
Now let  $\{\rho_i\}$ and $\{\rho'_i\}$ be a pair of 
  partitions  of unity on $X$ subordinate  to $\{U_i\}$, then there exists an even differential form $\eta (f, \{ \lambda_i\}, \{\rho_i\}, \{\rho'_i\} ) $ on $X$ such that 
\[
ch_{\check\sigma} (f, \{\lambda_i\}, \{\rho_i\} ) - ch_{\check\sigma} (f, \{\mu_i\}, \{\rho'_i\} ) = 
(d-H)\eta (f, \{ \lambda_i\}, \{\rho_i\}, \{\rho'_i\}).\]
\item Let  $\{\lambda_i <\mu_i\}$ be a pair of spectral cuts    for the family $\{f\circ\phi_i\}$ on
an open cover $\{U_i\}$ of $X$. Then there exists an even differential form $\eta (f, 
\{ \lambda_i\}, \{\mu_i\}, \{\rho_i\} ) $ on $X$ such that 
\[
ch_{\check\sigma} (f, \{\lambda_i\}, \{\rho_i\} ) - ch_{\check\sigma} (f, \{\mu_i\}, \{\rho_i\} ) = 
(d-H)\eta (f, \{ \lambda_i\}, \{\mu_i\}, \{\rho_i\}).\]
\item Let $f: P_{\sigma} \times [t_0, t_1] \to \Fred_*^{sa}$ be a $\PU(\cH)$-equivariant homotopy connecting $f_0 = f(\cdot, t_0)$ and  $f_1 = f(\cdot, t_1) $ such that there is a family of spectral cuts
$\{\lambda_i\}$   independent of $t\in [t_0,t_1]$ for $\{f( \phi_i,t)\}$ on
an open cover $\{U_i\times [t_0,t_1]\}$ of $X\times [t_0,t_1]$. Then there exists an even differential form $\eta (f_0, f_1, \{\lambda_i\},  \{\rho_i\}) $ on $X$ such that 
\[
ch_{\check\sigma} (f_0, \{\lambda_i\}, \{\rho_i\} ) - ch_{\check\sigma} (f_1, \{\lambda_i\}, \{\rho_i\} ) = 
(d-H)\eta (f_0, f_1,  \{\lambda_i\},  \{\rho_i\}).
\]
\end{enumerate} 
\end{proposition}
\begin{proof}    \begin{enumerate}
\item From the definition of  the  twisted differential  Chern character, we have
\[
ch_{\check\sigma} (f, \{\lambda_i\}, \{\rho_i\} )|_{U_j} = \sum_i (d-H) \rho_i \wedge e^{B_j}   ch(E_{ij}, \nabla_{ij})
\]
for the spectral cut  $\{\lambda_i\}$ and a partition of unity $ \{\rho_i\}$, 
and 
\[
ch_{\check\sigma} (f, \{\lambda_i\}, \{\rho'_i\} )|_{U_j} = \sum_i (d-H) \rho'_i  \wedge  e^{B_j} ch(E_{ij}, \nabla_{ij})
\]
for the spectral cut  $\{\lambda_i\}$ and a partition of unity $ \{\rho'_i\}$. Their  difference is given by
\[
\sum_i (d-H)   \big( (\rho_i- \rho'_i)  e^{B_j} ch(E_{ij}, \nabla_{ij})\big). 
\]
Note that $\sum_i    \big( (\rho_i- \rho'_i)  e^{B_j} ch(E_{ij}, \nabla_{ij})\big) $  is a globally defined  even
differential form  denoted by $\eta (f, \{ \lambda_i\}, \{\rho_i\}, \{\rho'_i\} ) $, as 
\[
\sum_i    \big( (\rho_i- \rho'_i)  e^{B_j} ch(E_{ij}, \nabla_{ij})\big)  
= \sum_i    \big( (\rho_i- \rho'_i)  e^{B_k} ch(E_{ik}, \nabla_{ik})\big) 
\]
over $U_j\cap U_k$ by the cocycle condition of $\{ e^{B_j} ch(E_{ij}, \nabla_{ij})\}$.  This verifies that
\[
ch_{\check\sigma} (f, \{\lambda_i\}, \{\rho_i\} ) - ch_{\check\sigma} (f, \{\mu_i\}, \{\rho'_i\} ) = 
(d-H)\eta (f, \{ \lambda_i\}, \{\rho_i\}, \{\rho'_i\}).\]

\item Let  $E_i $ be the vector bundle over $U_i$,  equal to the spectral subspace   defined by the  open    interval  $(\lambda_i, \mu_i)$  in the spectrum of $f(\phi_i(x))$.  Then we have 
\[
  F_{ij} \oplus  L_{ji} \otimes E_i   =   E_{ij} \oplus E_j 
  \]
on $U_{ij}.$   Equipped with the natural connections, we have
\[
 (F_{ij}, \nabla_{F_{ij}} )  \oplus  ( L_{ij}, A_{ij})  \otimes (E_i, \nabla_ {E_i}) \cong 
 (E_{ij}, \nabla_{E_{ij}} )  \oplus (E_j , \nabla_{E_j}).
\]
Hence,  
\[
 e^{B_j}  ch(F_{ij}, \nabla_{F_{ij}} )     = 
  e^{B_j}  ch(E_{ij}, \nabla_{E_{ij}} )   +  e^{B_j}  ch (E_j , \nabla_{E_j}) - 
  e^{B_i}  ch (E_i , \nabla_{E_i}).
  \]
Note that 
\begin{enumerate}
\item $ch_{\check{\sigma}} (f, \{\lambda_i\}, \{\rho_i\}  ) |_{U_j} = (d-H) \sum_i \rho_i  e^{B_j}  ch(E_{ij}, \nabla_{E_{ij}} )$.
\item $ch_{\check{\sigma}} (f, \{\mu_i\}, \{\rho_i\} ) |_{U_j} = (d-H) \sum_i \rho_i  e^{B_j}  ch(F_{ij}, \nabla_{F_{ij}} )$.
\item $ e^{B_j}  ch (E_j, \nabla_{E_j})=  \sum_i \rho_i e^{B_j}  ch (E_j , \nabla_{E_j})  $ is $d-H$ closed.
\end{enumerate}
Then the difference of the twisted differential 
Chern character form of $f$ with respect to a pair of  spectral cuts $\{\lambda_i < \mu_i\}$
is given by
\[
ch_{\check\sigma} (f, \{\lambda_i\}, \{\rho_i\}  ) - ch_{\check\sigma} (f, \{\mu_i\}, \{\rho_i\}  ) = 
(d-H)\eta (f, \{ \lambda_i\}, \{\mu_i\}, \{\rho_i\} ) 
\]
where  $\eta (f, \{ \lambda_i\}, \{\mu_i\}, \{\rho_i\}  ) =  \sum_i \rho_i e^{B_i}  ch (E_i, \nabla_{E_i})$. 
 \item Let $E_{ij}(t)$ be the finite rank vector bundle with the natural connection $\nabla(t)$  over $U_i\cap U_j$,  whose fiber at $x$ is spanned by the eigenvectors
 of $f(\phi_i(x), t)$ with eigenvalues in the open interval $(\lambda_i, \lambda_j)$. Then 
 $E_{ij}(t)$ is a continuous family of vector bundles over $U_i\cap U_j$ with natural connections
 satisfying
 \[
  (L_{kj}, A_{kj}) \otimes (E_{ij}(t), \nabla_{ij}(t) ) \oplus (E_{jk}(t), \nabla_{jk}(t)) = (E_{ik}(t), \nabla_{ik}(t)).
  \]
  Applying the 
 standard Chern-Weil argument to 
 the  odd     differential form $ch_{\check\sigma} (f, \{\lambda_i\} )$ on $M\times [t_0, t_1]$, 
 we have 
\[
 ch_{\check\sigma} (f_0, \{\lambda_i\}, \{\rho_i\} ) - ch_{\check\sigma} (f_1, \{\lambda_i\}, \{\rho_i\} ) = 
(d-H)\eta (f_0, f_1,  \{\lambda_i\},  \{\rho_i\}) 
\]
for $\eta (f_0, f_1,  \{\lambda_i\},  \{\rho_i\}) =  \disp{\int_{t_0}^{t_1}}  ch_{\check\sigma} (f,  \{\lambda_i\},  \{\rho_i\} )$. 
 \end{enumerate}
 \end{proof}

A   differential  twisted  $K^1$-cocycle   with the twisting $\check{\sigma} = (\cG_\sigma, \theta, \omega)$ given  by a lifting bundle gerbe  with connection and curving    is a quadruple 
\[
(f, \{\lambda_i\}, \{\rho_i\},  \eta)
\]
where    $f$  is a $\PU(\cH)$-equivariant  map  $\cP_\sigma \to \Fred_*^{sa}$, $\{\lambda_i\}$ is a spectral cut of $f$ with respect to an open cover $\{U_i\}$,  $\{\rho_i\}$  is a  
  partition  of unity  on $X$ subordinate  to $\{U_i\}$,    and  
   \[
   \eta \in
   \Omega^{even}(X ) /\Im (d-H),
   \]
where $\Im (d-H) $ denotes the image of  the twisted  de Rham  operator
$
d-H:  \Omega^{even}(X)\longrightarrow
\Omega^{odd}(X).
$

We introduce an equivalence relation on the set of $\check{\sigma}$-twisted differential $K^1$-cocycles,
generated by the following three elementary equivalence relations (See Proposition \ref{equivalence})
\begin{enumerate}
\item    $(f,  \{\lambda_i\},  \{\rho_i\},  \eta_0)$  and $(f, \{\lambda _i \},  \{\rho'_i\},  \eta_1)$  are  equivalent  if $\{\lambda_i\}$ is  a 
spectral cut  of $\{f\circ \phi_i \}$ associated to  local
trivializing sections $\{ \phi_i: U_i \to \cP_\sigma\}$, and  $\{\rho_i\}$ and $\{\rho'_i\}$ are  a pair of 
  partitions  of unity  on $X$ subordinate  to $\{U_i\}$,   with
\[
 \eta_0-\eta_1 + \eta (f, \{ \lambda_i\}, \{\rho_i\}, \{\rho'_i\})  = 0 \qquad \mod (d-H).
 \]
 \item  $(f,  \{\lambda_i\},  \{\rho_i\},   \eta_0)$  and $(f, \{\mu_i \},  \{\rho_i\}, \eta_1)$  are  equivalent if  
$\{\lambda_i \} $ and $\{ \mu_i\}$ are  a pair of spectral cuts    for the family of operators $\{f(\phi_i(x))\}$ on an open cover $\{U_i\}$ of $X$,  with
\[
 \eta_0-\eta_1 + \eta (f, \{ \lambda_i\}, \{\mu_i\}, \{\rho_i\}) = 0 \qquad \mod (d-H).
 \]
 \item $(f_0,  \{\lambda_i\},  \eta_0)$  and $(f_1, \{\lambda_i \},  \eta_1)$  are equivalent  if  there exists
$f: P_{\sigma} \times [t_0, t_1] \to \Fred_*^{sa}$,  a $\PU(\cH)$-equivariant homotopy connecting $f_0 = f(\cdot, t_0)$ and  $f_1 = f(\cdot, t_1) $ such that there is a family of spectral cuts
$\{\lambda_i\}$   independent of $t\in [t_0,t_1]$ for $\{f( \phi_i,t)\}$ on
an open cover $\{U_i\times [t_0,t_1]\}$ of $X\times [t_0,t_1]$, with
\[
 \eta_0-\eta_1 + \eta (f_0, f_1,  \{\lambda_i\},  \{\rho_i\})=  0 \qquad \mod (d-H).
 \] 
\end{enumerate}

\begin{definition}\label{diff:twistK1}  The differential twisted K-theory $\check{K}^1(X, \check{\sigma})$ with
a twisting  given by   $\check{\sigma}= (\cG_\sigma, \theta, \omega)$  is the space of  equivalence
classes of  differential twisted $K^1$-cocycles. The differential  Chern character  form
of a differential twisted $K^1$-cocycle $(f, \{\lambda_i\},  \{\rho_i\},  \eta)$ is given by
\ba\label{form:1}
ch_{\check{\sigma}} (f , \{\lambda_i\},  \{\rho_i\},  \eta) = ch_{\check{\sigma}} (f , \{\lambda_i\}, \{\rho_i\}) +  
(d-H) \eta.
\na
 \end{definition}

\begin{theorem} \label{twisted:1} 
 The  differential 
Chern character form  (\ref{form:1})
on  differential twisted $K^1$-cocycles    defines  the differential  Chern character  
\[
ch_{\check{\sigma}}: \check{K}^1(X, \check{\sigma}) \longrightarrow
 \Omega^{odd}(X,  d- H ), 
\]
whose image consists of  odd degree  differential forms  on  $X$,
 closed under $d-H$. The differential Chern character   
 induces a well-defined  twisted   Chern character  
 \[
 ch_{\check{\sigma}}: K^1(X,  \sigma ) \longrightarrow
 H^{odd}(X,  d- H ). 
\]
\end{theorem}

 \begin{proof}
  There is a natural forgetful map 
$  \check{K}^1 (X, \check{\sigma}) \longrightarrow
 K^1(X,  \sigma)$
which sends  an  equivalence class of $(f,  \{\lambda_i\}, \{ \rho_i\}, \eta)$ to
 a $\PU(\cH)$-equivariant   homotopy class of $f$.  Proposition \ref{equivalence} says that  the differential   twisted Chern character map 
 \[
 ch_{\check{\sigma}}:    \check{K}^1(X, \check{\sigma}) \longrightarrow
 \Omega^{odd} (X, d-H) 
  \]
induces the  twisted Chern character homomorphism
    \[
  Ch_{\check{\sigma}}:  K^1(X, \sigma) \longrightarrow
H^{odd} (X, H)
  \]
  such that the following  diagram  commutes
 \ba\label{comm:odd}
  \xymatrix{
    \check{K}^{1}  (X, \check{\sigma}) \ar[r]\ar[d]^{ch_{\check{\sigma}}}
      & K^1(X, \sigma)\ar[d]^{Ch_{\check{\sigma}}}   \\
  \Omega^{odd}_0(X,  d- H) \ar[r]  & H^{odd}(X, d- H),  }
\na
where $ \Omega^{odd}_0(X, d- H) $ denotes the image of 
the differential Chern character form homomorphism
$
ch_{\check{\sigma}}: \check{K}^1(X, \check{\sigma}) \longrightarrow
 \Omega^{odd}(X,  d- H ).$
\end{proof}

It is easy to see that horizontal homomorphisms in the commutative diagram (\ref{comm:odd}) are surjective.  In order to study the kernel of the forgetful homomorphism  
$
  \check{K}^1 (X, \check{\sigma}) \longrightarrow
 K^1(X,  \sigma), 
$
we need to introduce the even differential twisted K-theory   $\check{K}^0(X, \check{\sigma})$   and its differential Chern character.  We can define a  differential  twisted $K^0$-cocycle  to
be a $S^1$-family of   differential  twisted $K^1$-cocycles which passes through a
 differential  twisted $K^1$-cocycle that defines the zero element in $\check{K}^1 (X, \check{\sigma})$. 
  Then we have the following commutative diagram
 \[
 \xymatrix{
    \check{K}^{0}  (X, \check{\sigma}) \ar[r]\ar[d]^{ch_{\check{\sigma}}}
      & K^0(X, \sigma)\ar[d]^{Ch_{\check{\sigma}}}   \\
  \Omega^{ev}_0(X,  d- H) \ar[r]  & H^{ev}(X, d- H),  }
\] 
where $ \Omega^{ev}_0(X,  d-H) $ denotes the image of 
the differential  Chern character homomorphism
\[
ch_{\check{\sigma}}: \check{K}^0(X, \check{\sigma}) \longrightarrow
 \Omega^{ev}(X,  d- H ).
\]

\begin{theorem}\label{diff:twsitedK1}
 Denote by $K^0_{ \RR/\ZZ} (X, \check{\sigma})$ the kernel of the odd    differential Chern character 
$
ch_{\check{\sigma}}: \check{K}^1(X, \check{\sigma}) \longrightarrow
 \Omega^{odd}_0(X,  H), $
 then we have the following  two exact sequences:
 \ba\label{exact:1} 
 \xymatrix{
0\ar[r] & K^0_{ \RR/\ZZ}  (X, \check{\sigma})  \ar[r] & \check{K}^1(X, \check{\sigma}) \ar[r]^{ch_{\check{\sigma}}}& 
\Omega_0^{odd}(X, d-H)\ar[r]& 0,
 }\na
  \ba\label{exact:2}
 \xymatrix{
0\ar[r] &\disp{\frac{ \Omega^{ev}(X)}{\Omega_0^{ev}(X, d- H)}} \ar[r]  & \check{K}^1(X,\check{\sigma} )\ar[r] &  
K^1(X, \sigma) \ar[r]& 0,}
\na
 such that   the following  diagram commutes
 \ba\label{diagram:2}
  \xymatrix{&   0\ar[d] & \\
         & K^0_{ \RR/\ZZ} (X, \check{\sigma} ) \ar[rd] \ar[d]  &   \\
    0\to\disp{\frac{ \Omega^{ev}(X)}{\Omega_0^{ev}(X, d-H)}} \ar[r] \ar[r]\ar[dr]_{d-H} &
     \check{K}^{1}  (X, \check{\sigma}) \ar[r]\ar[d]^{\ ch_{\check{\sigma}}} 
      & K^1(X, \sigma)\ar[d]^{Ch_{\check{\sigma}}} \to 0 \\
         & \Omega_0^{odd}(X, d- H) \ar[r] \ar[d]  & H^{odd}(X, d- H ) \\
         &   0  &  }. \na
\end{theorem}
\begin{proof}  There is an exact sequence
\[ \xymatrix{
\disp{\frac{ \Omega^{ev}(X)}{\Im(d-H)  }} \ar[r]  & \check{K}^1(X,\check{\sigma} )\ar[r] &  
K^1(X, \sigma) \ar[r]& 0,}
\]
where  $\eta \in \disp{\frac{ \Omega^{ev}(X)}{\Im (d-H)  }} $  is mapped to 
a differential twisted $K^1$-class given by
\[
[(F, \{0\}, \{1\}, \eta)].
\]
Here $F $ is a   $\PU(\cH)$-equivariant map  from $\cP_\sigma$ to an orbit
of  the adjoint action of $\U(\cH) $ through a fixed  grading operator in  $\cH$ such that 
both eigenspaces  are infinite dimensional. 
 Note that $(F, \{0\}, \{1\}, \eta)$ is equivalent to $ (F, \{0\}, \{1\}, 0)$ if and only if
$\eta$  is a differential twisted Chern character form,  modulo  $\Im (d-H)$,   of an even differential twisted $K^0$-cocycle, that is 
\[
\eta   \in  \dfrac{  \Omega_0^{ev}(X, d-H) }{\Im (d-H)}.
\]
This gives rise to the exact sequence (\ref{exact:2}).   The verification of the commutative diagram 
(\ref{diagram:2}) is straightforward. 
\end{proof}

\begin{theorem}\label{diff:twsitedK0} Denote by $K^1_{ \RR/\ZZ} (X, \check{\sigma})$ the kernel of the even   differential Chern character 
$
ch_{\check{\sigma}}: \check{K}^0(X, \check{\sigma}) \longrightarrow
 \Omega^{ev}_0(X,  H), $
 then we have the following exact sequences:
 \ba\label{exact:3} 
 \xymatrix{
0\ar[r] & K^1_{ \RR/\ZZ}  (X, \check{\sigma})  \ar[r] & \check{K}^0(X, \check{\sigma}) \ar[r]^{ch_{\check{\sigma}}}& 
\Omega_0^{ev}(X, H)\ar[r]& 0,
 }\na
   \ba\label{exact:4}
 \xymatrix{
0\ar[r] &\disp{\frac{ \Omega^{odd}(X)}{\Omega_0^{odd}(X, H)}} \ar[r]  & \check{K}^0(X,\check{\sigma} )\ar[r] &  
K^0(X, \sigma) \ar[r]& 0,}
\na
such that   the following  diagram commutes
\ba\label{diagram:3} 
  \xymatrix{  &   0\ar[d] & \\
         & K^1_{ \RR/\ZZ} (X, \check{\sigma}) \ar[rd]  \ar[d]  &    \\
    0\to\disp{\frac{ \Omega^{odd}(X)}{\Omega_0^{odd}(X,d-  H)}} \ar[r]\ar[dr]_{d-H} &
     \check{K}^0  (X, \check{\sigma}) \ar[r]\ar[d]^{\ ch_{\check{\sigma}}} 
      & K^0(X, \sigma)\ar[d]^{Ch_{\check{\sigma}}} \to 0 \\
         & \Omega_0^{ev}(X, d- H) \ar[r] \ar[d]  & H^{ev}(X, d- H ) \\
         &   0  & . } \na
         \end{theorem}

\begin{remark}\label{locked:terms}  We  point out that commutative diagrams (\ref{diagram:2}) and (\ref{diagram:3})   of exact sequences are joined together by two interlocking  six-term exact sequences
\[
 \xymatrix{
K^1_{ \RR/\ZZ} (X, \check{\sigma})   \ar[r]&K^0(X, \sigma) \ar[r]^{Ch_{\check{\sigma}} }&   H^{ev}(X, d- H )   \ar[d]\\
H^{odd}(X, d- H )     \ar[u] & K^1(X, \sigma) \ar[l]^{Ch_{\check{\sigma}}}&
 K^0_{ \RR/\ZZ} (X, \check{\sigma})       \ar[l]
   }
 \]
and 
\[
 \xymatrix{
\disp{\frac{ \Omega^{odd}(X)}{\Omega_0^{odd}(X,d-  H)}}  \ar[r]^{d-H} & \Omega_0^{ev}(X, d- H)   \ar[r]&   H^{ev}(X, d- H )   \ar[d]\\
H^{odd}(X, d- H )     \ar[u] & \Omega_0^{odd}(X, d- H) \ar[l]^{Ch_{\check{\sigma}}}&
\disp{\frac{ \Omega^{ev}(X)}{\Omega_0^{ev}(X,d-  H)}}      \ar[l]^{d-H} 
   }.
   \] 
  In \cite{SimSul},  an analogue of the above interlocking  commutative diagram   in ordinary differential
  cohomology 
 is called the Character Diagram. 
 \end{remark}

\subsection{Classifying space  of   twisted K-theory} 
 In this subsection,  we will give a universal model of twisted K-theory. We start with the even case.
   As a model for the classifying space of $K^0$
we first choose $\Fred$,  the space of all bounded Fredholm operators in an infinite dimensional separable complex Hilbert space $\cH$.
Consider the twisted product 
\[
\mathfrak{G} = \PU(\cH) \ltimes  GL(\infty),
\] 
where the group $\PU(\cH)$ acts on the group $GL(\infty)$, of operators that are invertible  and $1+$ finite rank, by conjugation.  Thus the product
in $\mathfrak{G}$ is given by
$$ (g, f) \cdot (g',f')= (gg', f(g f' g^{-1}) ).$$
A principal $\mathfrak{G}$-bundle over $X$, locally trivial with respect to an open cover $\{U_i\}$,  
is defined by  transition functions
$
(g_{ij}, f_{ij}):  U_i\cap U_i  \longrightarrow \mathfrak{G} 
$
satisfying the usual cocycle condition which encodes both the cocycle relation for the transition functions $\{ g_{ij}\} $ of the
$\PU(\cH)$ bundle over $X$ and the twisted cocycle relation  for $\{f_{ij}\}$
\ba\label{twisted:cocycle}
  f_{ij}(\hat{g}_{ij} f_{jk} \hat{g}_{ij}^{-1}) = f_{ik}, 
\na
 which is independent of the choice of the lifting $ \hat{g}_{ij}$ of $g_{ij}$ to $\U(\cH)$.  

Next we construct a universal $\mathfrak{G}$ bundle $E$ over the  classifying space $B\mathfrak{G} = E/\mathfrak{G}.$ The total
space is just the Cartesian product $E= \cP \times \cQ,$ where $\cP$ is the total space of a universal $\PU(\cH)$ bundle over the base
$K(\mathbb{ Z}, 3)$ and $\cQ$
is the total space of a universal $GL(\infty)$ bundle over the base $\Fred^{(0)},$ the index zero component of the classifying space  of $K^0.$  The total space of $\cQ$ is
the set of pairs $(q,w)$ with $q$ a parametrix of a index zero Fredholm operator $w.$ This space is contractible. This follows from the observation that the pairs  can be parametrized  by $(q,t)$ with $w= (1 +t)q^{-1}$
and $t$ is an arbitrary finite rank operator, $q$ an arbitrary invertible operator. The right action of $GL(\infty)$ is given by $(q,w) \cdot a=
(qa,w)$ and is free.

The right action of $(g,a)\in\mathfrak{G}$ on  $E$ is defined by 
$$(p,(q,w))\cdot (g,a) = (pg, (g^{-1}qag, g^{-1} wg))$$
where $w$ is a Fredholm operator with parametrix $q.$ It is easy to see that the action is free, because $\PU(\cH)$ acts freely on the
first component and the second component involves right multiplication of the invertible operator $q$ by the group element $a\in GL(\infty).$

Now we can establish a  model for the classifying space for even  twisted K-theory. 

\begin{theorem}\label{universal:0}  Given a principal $\PU(\cH)$-bundle $\cP_\sigma$
 over $X$ with $\sigma: X \to K(\mathbb Z, 3)$ representing a non-torsion class in $H^3(X, \ZZ)$,  then  the  even twisted K-theory  $K^0(X, {\sigma})$  can be identified with  the
set of homotopy classes of maps $X \to B\mathfrak{G}$ covering the map $\sigma$.  
\end{theorem}
\begin{proof}
Let $f: \cP_\sigma \to \Fred $ be a $\PU(\cH)$-equivariant  family of Fredholm operators.  We can select an open cover $\{U_i\}$  of $X$ such that on each $U_i$  there is a local section  $\phi_i: U_i \to \cP_\sigma$ and  for each $i$ the index zero  Fredholm operators $f(\phi_i(x))$, $x\in U_i$ have a 
gap in the spectrum at some $\lambda_i \neq 0.$ Then over $U_i$ we have a finite rank vector bundle $E_i$ defined  by the spectral projection $f(\phi_i(x))^2 < \lambda_i^2.$  
As usual the transition functions are
$g_{ji} : U_{ji} \to \PU(\cH)$ 
with  $\phi_i(x)= \phi_j(x) g_{ji}(x)$
and the lifts $\hat{g}_{ij}$ of $g_{ij}$
to $\U(\cH)$ satisfy
$ \hat{g}_{ij} \hat{g}_{jk} \hat{g}_{ki} =\sigma_{ijk} Id  $
on triple intersections with $\sigma_{ijk}$ taking values in $U(1).$

 We may assume that $E_i$ is a trivial vector bundle over $U_i$  of rank $n_i$ by passing to a finer cover if necessary.  
 Choosing a trivialization of $E_i$  gives  a  $\ZZ_2 $ graded parametrix  $q_i(x)$ (an inverse up to
finite rank operators) for each $f(\phi_i(x))$, $x\in U_i$.   The operator $q_i(x)^{-1}$ is defined as the direct sum of the restriction
of $f(\phi_i(x))$ to the orthogonal complement of $E_i$ in $\cH$ and an isomorphism between the vector bundles $E^+_i$ and $E_i^-.$  Clearly
then $f(\phi_i(x)) q_i(x) =1$ modulo rank $n_i$ operators. 
 
For $x\in U_{ij}$ we have a pair of parametrices $q_i(x)$ and $ q_j(x) $ for $f(\phi_i(x))$
and $f(\phi_j(x))$  respectively. These are related by an invertible operator $f_{ij}(x)$ of type 
$1+$ finite rank and
$$  \hat{g}_{ij} q_j (x) \hat{g}_{ij}^{-1}  = q_i(x) f_{ij}(x).$$ 
The conjugation on the left hand side  by $\hat{g}_{ij}$  comes from the equivariance relation
\[
 f(\phi_j(x)) = f(\phi_i(x)g_{ij}(x)) = \hat{g}_{ij}(x)^{-1} f(\phi_i(x))\hat{g}_{ij}(x).
\] 

The system $\{f_{ij}\}$ does not quite satisfy the \v{C}ech cocycle relation   because of the
different local sections $\phi_i: U_i \to \mathcal{P}_{\sigma}$ involved. Instead, we have
 on $U_{ijk}$ 
$$\hat{g}_{jk} q_k \hat{g}_{jk}^{-1} = q_j f_{jk} = (\hat{g}_{ij}^{-1} q_i f_{ij} \hat{g}_{ij} )f_{jk}
= \hat{g}_{jk}( \hat{g}_{ik}^{-1} q_i f_{ik}\hat{g}_{ik} ) \hat{g}_{jk} ^{-1}.$$
Using the   relation $\hat{g}_{jk}  \hat{g}_{ik}^{-1} =\sigma_{ijk} \hat{g}_{ij}^{-1}$,  we get
\[
    \hat{g}_{jk} ( \hat{g}_{ik}^{-1} q_i f_{ik}\hat{g}_{ik} )  \hat{g}_{jk}^{-1} = \hat{g}_{ij}^{-1} q_i f_{ik} \hat{g}_{ij} 
\]
multiplying the last equation from right by $\hat{g}_{ij}^{-1}$ and from the 
left by $q_i^{-1}\hat{g}_{ij} $ one gets the twisted cocycle relation
\[
 f_{ij}(\hat{g}_{ij} f_{jk} \hat{g}_{ij}^{-1}) = f_{ik}, 
\]
 which is independent of the choice of the lifting $ \hat{g}_{ij}$.  
We may think of  (\ref{twisted:cocycle})   as defining an untwisted cocycle relation
for $\{(g_{ij}, f_{ij})\}$  in the twisted product 
$
\mathfrak{G} = \PU(\cH) \ltimes GL(\infty),
$
 In summary, this cocycle   $\{(g_{ij}, f_{ij})\} $ defines a principal
$\mathfrak{G}$ bundle over $X$ whose classifying map is a continuous map $X\to B\mathfrak{G}$. 
A homotopic  $\PU(\cH)$-equivariant  family of Fredholm operators gives rise to a homotopic
classifying map $X\to B\mathfrak{G}$.  
\end{proof}

\begin{remark}
The base space $B\mathfrak{G}$ is a fiber bundle over $K(\mathbb Z,3).$ 
The projection is defined by 
\[
\xi((p,(q,w))\mathfrak{G} ) = \pi(p),
\]
 where  $\pi: \cP \to K(\mathbb Z, 3)$ is the projection. The fiber $\xi^{-1}(z)$ at $z\in K( \mathbb Z,3)$ is isomorphic (but not canonically so) to the space $\Fred$ of
Fredholm operators; to set up the isomorphism one needs a choice of element $p$ in the fiber $\pi^{-1}(z).$ Thus a section of the bundle $B\mathfrak{G}$ over $K(\mathbb Z,3)$ is the same thing as a $\PU(\cH)$ equivariant map $\gamma$
 from $\cP$ to $\Fred$.  We use the notation $K^0(K(\mathbb Z, 3), \cP)$ for the twisted K-group of $K(\mathbb Z, 3)$ with $\cP$ being thought of as the twist. Then $\gamma$ is defining an element of
 $K^0(K(\mathbb Z, 3), \cP)$.
 This latter group is not known  (and of course the same is true for the
odd twisted K-group of $K(\mathbb Z,3)$).  The case of ordinary  K-theory on the Eilenberg-MacLane space $K(\mathbb Z, 3)$ is already 
complicated, but it is known that it is given by the integral cohomology of the space $K(\mathbb Q, 3),$ (Cf. \cite{AHod}).
\end{remark}

The twisted Chern character  on $K^0(X, \mathcal{\sigma})$ can also be constructed  by choosing a connection $\nabla$ on a principal $\mathfrak{G}$ bundle  over $X$ associated to the cocycle $\{(g_{ij}, f_{ij}\}$.   Locally, on a 
trivializing open cover  $\{U_i\}$ of $X$, we can lift the connection to a connection taking values in the Lie algebra $\hat{\mathbf{g}}$ of the central
extension $U(H) \times GL(\infty)$ of $\mathfrak{G}.$  Denote by $\hat F_{\nabla}$ the curvature of this connection. On the overlaps $U_{ij}$ 
the curvature satisfies a twisted relation
$$ \hat F_{\nabla,j} =  Ad_{(g_{ij}, f_{ij})^{-1}}  \hat{F}_{\nabla,i}  + g^*_{ij} c,$$
where $c$ is the curvature of the canonical connection $\theta$ on
 the principal $U(1)$-bundle $\U(\cH)\to \PU(\cH).$ 

Since the Lie algebra $\mathbf{u}(\infty) \oplus \CC$ is an ideal in the Lie algebra of $U(\cH) \times GL(\infty),$ the projection $F'_{\nabla,i}$ of the curvature
$\hat F_{\nabla,i}$ onto this subalgebra transforms in the same way as $\hat F$ under change of local trivialization. It follows that for a $\PU(\cH)$-equivariant map $f: \cP_\sigma \to \Fred$,  we can define a twisted 
Chern character form of $f$ as
$$  ch_{\check{\sigma}} (f, \nabla )   = e^{B_i} \text{tr}\, e^{ F'_{\nabla,i}/2\pi i},$$ over $U_i$. 
Here the trace is well-defined on $\mathbf{gl}(\infty)$ and on the center $\CC$ it is defined as the coefficient of the unit operator.  
Note that   $ch_{\check{\sigma}} (f, \nabla )$  is globally defined and $(d-H)$-closed
$$(d-H) ch_{\check{\sigma}} (f,  \nabla )  =0,$$ 
and depends on the differential twisting $
\check{\sigma} = (\cG_\sigma, \theta, \omega ),$
 and  a choice of  a connection $\nabla$ on a principal $\mathfrak{G}$ bundle   over $X.$

 We now give a short description of the classifying space in odd twisted K-theory.
First let $\cH = \cH_+ \oplus \cH_-$ be a decomposition into isomorphic subspaces.
Operators on $\cH$ can be written as two by two matrices of operators with respect to this decomposition.
The group
$\U_{res}=\U_{res}(\cH)$ is the group of unitary operators in $\cH$ with Hilbert-Schmidt off-diagonal blocks (for more 
on this definition see \cite{PrSe}).
In the non-twisted case an element of $K^1(X)$ is given by a
homotopy class of maps $X\to U(\infty).$  Since $U(\infty)$ is the
base space of a universal $\U_{res}$ bundle, this is the same thing as
giving an equivalence class of $\U_{res}(\cH)$ bundles over $X.$  

\begin{remark} \label{rem:CM}  The universal $\U_{res}$ bundle $Q$ can be constructed as follows. By
Bott periodicity, $\U_{res}$ is homotopy equivalent \cite{PrSe} to the group
$\Omega U(\infty)$ of smooth based loops in $U(\infty).$ The universal
$\Omega U(\infty)$ bundle $Q$ over $U(\infty)$ is simply the space of
smooth  paths $f:[0,1]\to U(\infty)$ with $f(0)=1$ and  $f^{-1}df$ periodic. The right
action of $\Omega U(\infty)$ on $Q$ is just the pointwise
multiplication of paths, see Appendix 2 in \cite{CM}.
\end{remark}

A $\U_{res}$ bundle over $X$ can be given in terms of local transition
functions $\phi_{ij} : U_{ij} \to \U_{res}.$ Next we need to twist
this construction by the transition functions $g_{ij}$ of the $\PU(\cH)$
bundle over $X.$ In the odd case for any $g\in \PU(\cH)$ we had an automorphism 
$Aut_g$ of $U(\infty)$ given by a conjugation by any unitary operator
covering $g\in \PU(\cH).$ In the case of $\U_{res}$ we have to be a
little more careful since conjugation by a unitary operator is not in
general an automorphism of $\U_{res}.$ However, we recall that
$\U_{res}$ is a classifying space for even K-theory and is homotopy
equivalent to the space of bounded Fredholm operators in any separable infinite dimensional Hilbert
space. Actually, the homotopy equivalence is given explicitly as the
map $g\mapsto a,$ where 
$$ g= \left( \begin{matrix} a&b\\ c&d \end{matrix} \right)$$ 
in a decomposition $\cH=\cH_+\oplus \cH_-,$  \cite{PrSe}, \cite{Wur}. The block $a$ is always a Fredholm operator and one can construct a homotopy
inverse to this map. This gives directly the right way to define
automorphisms
of $\U_{res}$ by elements of $\PU(\cH_+),$ compatible with automorphisms
of the space of Fredholm operators: For any $g\in \PU(\cH_+)$ one selects 
$\hat g\in \U(\cH_+)$ and considers this as a unitary operator in $\cH$ 
by 
$$g \mapsto \left(\begin{matrix} \hat g & 0 \\0 & \hat
  g\end{matrix}\right).$$ 
  Thus, in the same way as in the even case, we can think of an element in
the twisted K-theory group $K^1(X, \sigma)$ defined by transition functions 
$h_{ij} : U_{ij} \to \mathfrak{H},$ where $\mathfrak{H}$ is the group 
$$\mathfrak{H} = \PU(\cH_+) \times_{\PU(\cH_+)} \U_{res}$$ 
with multiplication $(g,h)(g',h') = (gg', h (gh'g^{-1}).$

More precisely, using the concrete realization of the universal $\U_{res}$ bundle in  
Remark  \ref{rem:CM},
we can choose local sections $q_i: U_i \to Q$ by choosing for each $x\in U_i$ a path $q_i(x)$ in $U(\infty)$ joining 
the neutral element to the end point $f(\phi_i(x)),$ where we have taken $f$ as the equivariant function on $\cP_{\sigma}$ taking values
in the model $U(\infty)$ for the classifying space $\Fred_*^{sa}$ of odd K-theory. 
 On the intersection $U_{ij}$ we then have 
$$ \hat g_{ij}(x)  q_j(x) \hat g_{ij}^{-1}(x)  =  q_i(x)   \cdot f_{ij}(x),$$
where $f_{ij}(x)$ is an element of the loop group $\Omega U(\infty),$ which by the Remark gives an
element in $\U_{res}.$ From this transformation rule we obtain the  group multiplication law above  for the transition
functions $h_{ij}=(g_{ij}, f_{ij})$ taking values in the group $\mathfrak{H}.$

As in the even case, one can check that the classifying space
$B\mathfrak{H}$ is a fiber bundle over the base $K(\mathbb Z, 3)$ with
$\Fred_*^{sa}$ (or equivalently, $U(\infty)$) as the fiber. It is
  written as  $B\mathfrak{H} = \cP \times_{\PU(\cH_+)} \Fred_*^{sa},$
  where now  $\cP$ is the universal $\PU(\cH_+)$ bundle over $K(\mathbb
  Z, 3).$

\begin{theorem}\label{universal:1}  Given a principal $\PU(\cH_+)$-bundle $\cP_\sigma$
 over $X$ defined by $\sigma: X \to K(\mathbb Z, 3)$,  the odd  twisted K-group  $K^1(X, {\sigma})$ is the set of homotopy classes of maps $X \to B\mathfrak{H}$ covering the map $\sigma$.  
\end{theorem}

 \begin{example} \label{example}
   Let $E$ be a real Euclidean vector bundle of   rank $2n$  over $X$.   Then there is a lifting bundle
gerbe  $\cG_{W_3(E)}$ over $X$, called
the Spin bundle gerbe,   associated to  the frame bundle $SO(E)$ and the 
 flat $U(1)$-bundle over  $SO(2n)$:
\[
U(1) \to Spin^c(2n) = Spin(2n)\times_{\ZZ_2} U(1) \longrightarrow  SO(2n).
\]
Choose a local trivialization of $E$ over an open cover $\{U_i\}$ of $X$. Then the transition
functions 
\[
g_{ij}:  U_i\cap U_j \longrightarrow SO(2n).
\]
define an element in $H^1(X, \underline{SO(2n)})$ whose image under the Bockstein  exact sequence
\[
H^1(X, \underline{Spin(2n)}) \to H^1(X, \underline{SO(2n)}) \to H^2(X, \ZZ_2)
\]
is the second Stieffel-Whitney class $w_2(E)$ of $E$.  The corresponding $\PU(\cH)$-bundle
is defined by the following twisting 
\[
W_3(E):=W_3\circ \nu_E: X \longrightarrow  \BSO(2n)  \longrightarrow   K(\ZZ, 3), 
\]
where $\nu_E$ is the classifying map of $E$ and $W_3$ is the classifying map
of the principal $U(1)$-bundle $\BSpin^c (2n) \to \BSO (2n)$.

Note that the Spin bundle gerbe can be equipped with a flat connection and trivial curving. 
Denote also by $w_2(E)$ the corresponding bundle gerbe 
with a flat connection and a trivial curving.   With respect to a good cover $\{U_i\}$
of $X$ the  differential twisting $w_2(E)$ defines a Deligne
cocycle 
$(\sigma_{ijk}, 0, 0)$
with  trivial local $B$-fields, here $\sigma_{ijk}  = \hat{g}_{ij} \hat{g}_{jk}\hat{g}_{ki}  $ where $\hat{g}_{ij}: U_{ij} \to Spin^c (2n)$  is  a lift of $g_{ij}$.

A  differential twisted K-class in $\check{K}^0(X,  w_2(E))$   can be represented by
a Clifford bundle, denoted $\cE$, equipped with a Clifford connection, 
where $E$ is also equipped with a $SO(2n)$-connection. 
Locally, over each $U_i$  we let
$\cE|_{U_i} \cong S_i\otimes \cE_i $
where $S_i$ is the local fundamental spinor bundle associated to $E|_{U_i}$ with the standard Clifford  action
of $\Cl(E|_{U_i})$ obtained from the fundamental representation of $Spin(2n)$.  Then $\cE_i$ is a complex vector bundle
over $U_i$ with a connection $\nabla_i$  such that
\[
ch(\cE_i, \nabla_i)  = ch(\cE_j, \nabla_j). 
\]
Hence, our construction of a differential  Chern  character  implies that  the twisted Chern character 
\[
Ch_{w_2(E)}: K^0(X, W_3(E)) \longrightarrow  H^{ev} (X) 
\]
is given by  sending $[\cE]$ to  $\{[ch(\cE_i, \nabla_i)] = ch(\cE_i)\}$.  If $E$ is equipped with a $Spin^c$ structure
whose determinant line bundle is $L$,  there is a canonical
isomorphism 
\[
K^0(X) \longrightarrow  K^0(X, W_3(E)), 
\]
given by $[V] \mapsto [V\otimes S_E]$ where $S_E$ is the associated spinor bundle of $E$. Then we have
\[
Ch_{w_2(E)} ([V\otimes S_E]) =  e^{\frac{c_1(L) }{2}} ch ([V]),
\]
where $ch([V])$ is the ordinary Chern character of $[V] \in K^0(X)$. 

If $X$ is an even dimensional
Riemannian manifold, and $TX$ is equipped with the Levi-Civita
connection,  we can identify $K^0(X,  W_3(E))$ with  the 
Grothendieck group of Clifford modules, denoted  $K^0(X, \Cl(TX))$, of the bundle of Clifford algebras $\Cl(TX)$. Then  
\ba\label{TX}
Ch_{w_2(X)} ([\cE]) =  ch(\cE/S)
\na
where $ch(\cE/S)$ is the relative  Chern character of the Clifford module $\cE$ constructed
in Section 4.1 of  \cite{BGV}.   
 In particular, if $X$ is equipped with a $Spin^c$ structure (Cf. \cite{LM}), let its 
canonical class be  $c_1$, and its complex spinor bundle associated to
the fundamental representation of $Spin^c(2n)$ be $S(X)$. Then  any Clifford module  $\cE$ can be 
written as $V\otimes S(X)$, whose twisted Chern character  is given by
\[
Ch_{w_2(X)} (\cE)  = e^{\frac{c_1}{2}} ch (V),
\]
where $ch(V)$ is the ordinary Chern character.  Note that the vector bundle  $V$ in the decomposition 
$\cE =V \otimes S(X)$ depends on the choice of $Spin^c$ structure, but 
$e^{\frac{c_1}{2}} ch (V)$ depends only on $\cE$.
\end{example}

 \section{Riemann-Roch theorem in twisted K-theory}

  A smooth map $f: X\to Y$ is called K-oriented
if $TX \oplus f^*TY$ is equipped with a Spin$^c$ structure which is determined by a choice
of $c_1\in  H^2(X, \ZZ)$ 
such that
\[
w_2(X)  -  f^*w_2(Y)  =  c_1 \mod 2 .
\]
The  push-forward map $f_{!}^{c_1}: K^*(X) \to K^{*-d(f)}(Y)$, also called the Gysin homomorphism in K-theory,  is well-defined, here $d(f) = (\dim Y-\dim X)\mod2$. 
The  Riemann-Roch theorem (Cf. \cite{AtiHir}) is  given by the following formula
  \ba\label{AH}
  Ch\bigl(f_{!}^{c_1}(a)\bigr) \hat{A}(Y) =  f_*^H\bigl(Ch(a)e^{\frac{c_1 }{2}}\hat{A}(X)\bigr),
  \na
  where   $a\in K^*(X)$,  and $f_*^H$ is the Gysin homomorphism in ordinary cohomology theory,
  $Ch$ is the Chern character for ordinary complex K-theory,  $\hat{A}(X)$ and $\hat{A}(Y)$ are the
  A-hat classes (which can be expressed in terms of the Pontrjagin classes) of $X$ and $Y$ respectively.
 
  Let $W_3(f)$
     represent  the image of $f^*(w_2(Y))- w_2(X)   $
   under the Bockstein homomorphism
 $H^2(X, \ZZ_2) \longrightarrow H^3(X, \ZZ).$
In \cite{CW2}, we showed that there is a natural push-forward map (\ref{push:forward})   \[
 f^K_!:   K^i \bigl(X, f^*\sigma  + W_3(f)\bigr) \longrightarrow K^{i+ d(f)}(Y, \sigma),
 \]
 associated to 
any differentiable map $f: X\to Y$ and any $\sigma: Y \to K(\ZZ, 3)$.
    The push-forward map is constructed by choosing a closed embedding 
   $\iota: X\to Y \times \RR$ such that $f= \pi\circ \iota$ where $\pi$ is  the projection.  The push-forward is functorial in the sense that
   \[
   f_!^K =   \pi_!^K\circ \iota_!^K,
   \]
   where the push-forward
   map $\pi_!^K$ is given by the Bott periodicity of  twisted K-theory.  So we may assume that
   $f: X\to Y $ is a closed embedding  and    $X$ and $Y$ are equipped with compatible Riemannian metrics.

   In this Section we will establish the Riemann-Roch theorem in twisted K-theory for any closed
   embedding $f:X\to Y$ with a twisting $\sigma: Y \to K(\ZZ, 3)$ (the classifying map for 
   $\cP_\sigma$).    We do this in two stages, represented by Theorems 4.1 and 4.4.
   Theorem 4.1 is a special case (that we need in this paper) which we will prove by a method that may be easily modified
   to give the general case (Theorem 4.3).
   Recall the notation
  $ \check{\sigma} = (\cG_\sigma, \theta, \omega)$
    for a differential twisting and $H$ for  the  normalized curvature of  $\check{\sigma}$.   We use $w_2(f)$ to denote the Spin bundle gerbe  with
    a trivial connection and a trivial curving   on $X$ associated to the normal bundle $N_f$ of the embedding $f$ (Cf. Example \ref{example}). 
   
      We first   assume that $f: X\to Y $ is K-oriented, that is,    
 the normal  bundle 
$ N_f \cong f^*TY /TX $
   of $f$ is equipped with a $Spin^c$ structure.        Under this assumption, we have the following
   Riemann-Roch theorem  in twisted K-theory for a K-oriented embedding $f: X\to Y$ and $[f^*\sigma]=0$.  For simplicity, we assume that $\dim Y- \dim X = 0 \mod 2$, otherwise, we can replace $Y$ by
   $Y\times \RR$.
   
 \begin{theorem}\label{RR:1}  Let $f: X\to Y$ be a closed embedding whose normal bundle is equipped
 with a $Spin^c$ structure. Assume that $\dim Y- \dim X =2n$ and  $f^*[\sigma]  =0$. 
   Let   $a\in K^0(X, f^*\sigma)$, and $c_1(f)$ be the canonical class of the $Spin^c$ structure
  on $N_f$. Let   $f_*^H$ be the Gysin homomorphism in twisted cohomology theory,  and $Ch_{\check{\sigma}}$ and $Ch_{f^*\check{\sigma}}$  be
   the twisted Chern characters   on $K^0(Y,\sigma)$ and $K^0(X, f^*\sigma)$
   respectively.  
 Then 
\[
Ch_{\check{\sigma}} \bigl(f_{!}^{K}(a)\bigr) \hat A (Y)    =  f_*^H\bigl(Ch_{f^*\check{\sigma}} (a)e^{\frac{-c_1(f) }{2}}\hat{A}(X)\bigr).
\]
  \end{theorem}
 \begin{proof}
  Choose a local  trivialization of $\cP_\sigma$ which defines  a degree two Deligne   cocycle  $ (\sigma_{ijk}, A_{ij}, B_i)$ on $Y$  associated to $\check{\sigma} $ with respect to
  an open cover $\{U_i\}$ and transition functions
$  g_{ij}:  U_i\cap U_j \longrightarrow  \PU(\cH).$
   The conditions  $  f^*[\sigma] =[W_3(f) ]=0$  imply  that   there is a  homotopy commutative diagram
   (Cf. \cite{Wang})
\ba\label{twisted-spinc}
\xymatrix{X \ar[d]_{f} \ar[r]^{\nu_f} &
\BSO (2n)
 \ar@2{-->}[dl]_{\eta} \ar[d]^{W_3} \\
Y \ar[r]_\sigma  & K(\ZZ, 3). } 
\na
Here  $\nu_f$ is   the classifying map of  the   normal bundle of $f$,  
$W_3$ is the classifying map of the principal 
$BU(1)$-bundle $\BSpin^c (2n) \to \BSO (2n)$  and   $\eta$ is a homotopy
between $W_3 \circ \nu_f $ and $\sigma \circ f $.

The  homotopy commutative diagram
  (\ref{twisted-spinc}) and the $Spin^c$ structure on $N_f$ define a 
  trivialization of $f^*\cP_\sigma$ given by a section $\psi: X \to f^*\cP_\sigma$.
  With respect to the  induced open cover from $\{U_i\}$ and the 
  trivialization of $f^*\cP_\sigma$, the transition functions $\{g_{ij}\}$ can be written as 
  \[
  g_{ij} = h_i h_j^{-1}
  \]
  for $h_i: \tilde{U}_i = X\cap U_i \to \PU(\cH)$.  Then the pull-back differential 
twist   $f^* \check{\sigma} $  defines a degree 2 Deligne cocycle
  \[
   (1, A_j-A_i, B_i)
   \]
   such that $B_i-dA_i$ is a globally defined 2-form on $X$, denoted by $B_X$, and $dB_X = f^*H$ on $X$. Note that such a  Deligne cocycle  can be extended to a tubular neighborhood of $X$ in $Y$
   with the 2-form denoted by $B_f$.

  Given a  twisted K-class in $K^0(X, f^* \sigma)$, represented by a locally defined
  map $\psi_i: \tilde{U}_i \to \Fred (\cH)$, we have 
  \[
  h_i^{-1} \psi_i h_i =  h_j^{-1} \psi_j h_j,
  \]
  which implies that $\{  h_i^{-1} \psi_i h_i\}$ defines a continuous map 
  $\psi:  X\to \Fred (\cH)$. Our construction of the differential twisted Chern character 
  shows that
  \[
  ch_{f^* \check{\sigma} } (\{ \psi_i\} ) = \exp (B_X) ch (\psi)
  \]
  where $ch$ is the ordinary Chern character form of $\psi$. Hence we 
  have established the existence of
  the following commutative diagram relating the ordinary Chern character  on $K^0 (X )$  and
  the twisted Chern character on $K^0 (X, f^*\sigma )$ respectively
  \ba\label{ch:1}
  \xymatrix{
  K^0 (X ) 
 \ar[rr]^{\cong}  \ar[d]^{Ch}   
 & &  K^{0}(X, \ f^*\sigma )  \ar[d]^{Ch_{f^*\check{\sigma} }}  \\
 H^{ev} (X, d)  \ar[rr]^{\exp (B_X)}_{\cong} &&  H^{ev} (X, d- f^*H) .
  }
 \na
 Note that the above arguments also work for $N_f$  which is identified with the tubular neighborhood
 of $X$ in $Y$. Denote by $\iota: N_f \to Y$ the open embedding, so we have a similar 
 commutative diagram relating the ordinary Chern character  on $K^0 (N_f )$  and
  the twisted Chern character on $K^0 (N_f , \iota^*\sigma )$ respectively
  \ba\label{ch:2}
 \xymatrix{
  K^0 (N_f ) 
 \ar[rr]^{\cong}  \ar[d]^{Ch}   
 & &  K^{0}(N_f,   \iota^*\sigma )  \ar[d]^{Ch_{\iota^*\check{\sigma} }}  \\
 H^{ev} (N_f, d)  \ar[rr]^{\exp( B_f) } _{\cong}&&  H^{ev} (N_f, d- \iota^*H) .
  }
 \na

For the open embedding $\iota: N_f \to Y$, the twisted Chern characters commute with the
push-forward map $\iota_!^K$, so the following diagram commutes
\ba\label{ch:3}
  \xymatrix{
  K^{0}(N_f, \iota^*\sigma )   
 \ar[rr]^{\iota^K_!}  \ar[d]^{Ch_{\iota^*\check{\sigma} }   }
 & &  K^{0}(Y,  \sigma )  \ar[d]^{Ch_{ \check{\sigma} }}  \\
 H^{ev} (N_f,  d- \iota^*H)  \ar[rr]_{\iota_!^H} &&  H^{ev} (Y, d-  H) .
  }
 \na

Introduce the zero section    $o: X\to N_f$,  
so that $f=\iota\circ o$, and the Riemann Roch theorem for $o$
gives  what we call the 
`Chern character defect diagram'
 \ba\label{non-commu}
  \xymatrix{
  K^0 (X ) 
 \ar[rr]^{o_!^K}  \ar[d]^{Ch}   
 & &  K^{0}(N_f  )  \ar[d]^{Ch }  \\
 H^{ev} (X, d)  \ar[rr]_{o_! }^{\cong} &&  H^{ev} (N_f, d ), 
  }
\na
  which  is not commutative.  Here the Gysin homomorphism $o_!$ is the
  Thom isomorphism.   The non-commutativity of (\ref{non-commu})  is  described  by the following
  formula
 \ba\label{ch:4}
  Ch\bigl( o_!^K (a)  \bigr) \hat A (N_f) =  o_!^H \bigl(Ch (a)  e^{\frac{-c_1(f)}{2}} \hat{A}( X) \bigr), 
  \na
for any $a\in K^0(X)$ where the homomorphism $o_!^H$ on twisted cohomology theory
is 
\[
o_!^H = \exp (B_f) \circ   o_! \circ \exp(-B_X) : H^{ev}(X, d-f^*H) \longrightarrow
 H^{ev} (N_f,  d- \iota^*H) .
 \]
 The Gysin homomorphism $f_!^H$ in twisted cohomology theory
is given by $\iota_!^H  \circ o_!^H$.

  The commutative  diagrams (\ref{ch:1}),
   (\ref{ch:2}),  (\ref{ch:3}) and the formula (\ref{ch:4}) imply that
   \[
  Ch_{\check{\sigma}} \bigl(f_{!}^{K}(a)\bigr) \hat A (Y)    =  f_*^H\bigl(Ch_{f^*\check{\sigma}} (a)e^{\frac{-c_1(f) }{2}}\hat{A}(X)\bigr),
\]
for any  $a\in K^0(X, f^*\sigma)$. Here we have applied the identity $\hat A(N_f) = 
\iota^* \hat A (Y)$. 
  \end{proof}

 \begin{remark} In ordinary K-theory, for  a closed embedding $f:X\to Y$  whose normal bundle is equipped with a $Spin^c$ structure $c_1(f)$, then $f$ is $-c_1(f)$-oriented as in \cite{AtiHir}, and 
 Theorem \ref{RR:1} agrees with the Atiyah-Hirzebruch's Riemann-Roch theorem. 
 \end{remark}

  \begin{remark}\label{remark:RR} Under the assumptions of Theorem   \ref{RR:1},
  for a  closed  Riemannian embedding $f: X\to Y$ whose normal bundle is equipped
   with a Hermitian  structure,  we have an integral   version
   of the Riemann-Roch Theorem \ref{RR:1}. To describe this, 
 let $E$ be an Hermitian vector bundle with  Hermitian connection $\nabla^E$ over $X$. Let
$\nabla^{TX}$ and $\nabla^{TY}$  denote the Levi-Civita
connections on $TX$ and $TY$ respectively, and  $\nabla_f$ be an Hermitian connection on $N_f$. 
Then
\ba\label{diff:RR}
 \int_Y ch_{\check{\sigma}} \bigl(f_{!}^{K}(E, \nabla^E)\bigr) \hat A (Y, \nabla^{TY})    = 
 \int_X ch (E, \nabla^E) e^{\frac{-c_1(N_f, \nabla_f) }{2}}\hat{A}(X, \nabla^{TX}),
\na
The equality (\ref{diff:RR}) follows from our construction of  
the twisted Chern character form and the differential version of 
the  diagrams (\ref{ch:1}),
   (\ref{ch:2}),  (\ref{ch:3}) and the formula (\ref{ch:4}).
\end{remark}

 The proof of Theorem \ref{RR:1} and arguments in Example \ref{example} 
 can be adapted to give the following general Riemann-Roch theorem in twisted K-theory. We leave the proof to dedicated readers.

 \begin{theorem}\label{RR:2}  Let $f: X\to Y$ be a closed embedding.
  Assume that $\dim Y- \dim X =2n$. Then 
\[
  Ch_{\check{\sigma}} \bigl(f_{!}^{K}(a)\bigr) \hat{A} (Y)   =  f_*^H\bigl(Ch_{f^*\check{\sigma}+w_2(f)} (a)\hat{A} (X)\bigr), 
\]
for any $a\in K^0 (X,  f^*\sigma +  W_3(f) ) $. 
 Equivalently, the following diagram commutes
     \ba\label{RR:diagram:2}
    \xymatrix{
  K^0 (X,  f^*\sigma +  W_3(f) ) 
 \ar[rr]^{f^K_!}  \ar[d]_{Ch_{ f^*\check{\sigma}+w_2(f) } (-) \hat A(X) }   
 &&   K^{0}(Y, \sigma)  \ar[d]^{Ch_{\check{\sigma} }  (-) \hat A(Y)}  \\
 H^{ev} (X, f^*H)    \ar[rr]_{f^H_*} &&  H^{ev} (Y,  H) .
  }
   \na
   In particular, if $\sigma=W_3(Y)$, then  the Riemann-Roch formula is given by 
   \[
  Ch_{w_2(Y)} \bigl(f_{!}^{K}(a)\bigr) \hat{A} (Y)   =  f_*^H\bigl(Ch_{w_2(X) } (a)\hat{A} (X)\bigr).
\]
for any $a\in K^0 (X,    W_3(X) ) $. 
  \end{theorem}
\vspace{2mm}

\section{Applications}

 In this Section, we apply our differential twisted K-theory and Riemann-Roch  theorem to investigate
twisted K-theory of simply connected simple Lie groups. 
 From the result of Douglas (\cite{Dou}) and Braun (\cite{Braun}), we know that the twisted K-theory of a simply connected simple Lie group of rank $n$ is an exterior
algebra on $(n-1)$ generators tensor   a cyclic group.  The order of the cyclic group can be
described in terms of the dimensions of irreducible representations of
$G$.   
So the twisted Chern characters are all zero for the twisted K-theory of a simply connected simple Lie group. We will apply
twisted Chern character forms to study these torsion elements in twisted K-theory.

Let $G$ be a simply connected compact simple Lie group of rank $n$ and $T$ be a maximal torus of $G$. The sets of
simple roots and fundamental  weights of $G$ with respect to $T$ are denoted by
\[
\{\a_1, \a_2 \cdots, \a_n\}, \qquad \{\l_1, \l_2 \cdots, \l_n\}
\]
respectively.  Let $\bf k$ denote the twisting $G\to K(\ZZ, 3)$
corresponding to the integer $k$ in $H^3(G, \ZZ)\cong \ZZ$ with the corresponding 
$\PU(\cH)$ bundle $\cP_{\bf  k}$.

 In \cite{CW1}, a canonical trivialization of the associated bundle
gerbe $\cG_{\bf  k}$ over certain conjugacy classes in $G$ is constructed. Those conjugacy classes are  
diffeomorphic to $G/T$, called symmetric D-branes in \cite{BouDaw}, \cite{FFFS}, \cite{FS}, \cite{Gaw} and labeled by dominant weights of level $k$.  For each simple root $\a_i$, the corresponding  subgroup 
   $SU(2)_{\a_i}$  intersects with these D-branes  at 
$S^2_{\a_i}$, which splits $SU(2)_{\a_i}$ into two 3-dimensional  balls $D_i^+$ and $D_i^-$.  
Denote by $C_i = C_i^+\cup C_i^-$ the CW complex obtained by attaching these  two 3-dimensional  cells  $D_i^+$ and $D_i^-$ to $G/T$. 

Let $\iota_i: C_i \to G$ be the inclusion map. Then the restriction  $\iota_i^*{\bf  k} = {\bf  k} \circ \iota_i: C_i \to K(\ZZ, 3)$
defines the integer $k$ in  $H^3(C_i,  \ZZ)\cong \ZZ$. 
Note that $C_i^+\cap C_i^-  =  G/T$, and denote by $\iota_i^{\pm}: G/T \to C_i^\pm$ and 
$\iota: G/T \to G$  the
inclusion maps. The twisting  $k$ is trivial when restricted to $C_i^\pm$,  and the 
two  trivializations on $G/T$ differ by  a line bundle $L_{k\l_i}$ whose
first Chern class is given by $k\l_i \in H^2 (G/T)$.  We have the following Mayer-Vietoris sequence
 \ba\label{C_i}
 \xymatrix{
K^1 (C_i^+  )\oplus K^1( C_i^- )\ar[r]& 0 \ar[r]&   K^0 (C_i, {\bf  k} )
  \ar[d]\\
K^1 (C_i, {\bf  k} ) \ar[u] & K^0  (G/T ) \ar[l]^{\delta} &
 K^0  (C_i^+ ) \oplus K^0(C_i^-)     \ar[l]
 }
 \na
 where the map $K^0  (C_i^+ )   \oplus K^0(C_i^-)    \to  K^0  (G/T )  $ sends
 $([E_1] , [E_2]) $ to $[E_1] - [ L_{k\l_i}\otimes E_2]$. 
This implies that the push-forward map
\[
\iota_!^K:  K^0(G/T) \longrightarrow   K^{\dim G}(G, {\bf  k}) 
\]
 satisfies $\iota_!^K ([E]) = \iota_!^K ([L_{k\l_i}\otimes E])$ for any vector bundle $E$ over $G/T$.   
 
 Choose a  gerbe connection and a curving on  $\cP_{\bf  k}$ to get a differential
 twisting $\check{{\bf  k}}$  whose normalized curvature is $H$.   
As  the twisted Chern characters are zero on $K^*(G, \bf  {k})$, we see that  
\[
K^*(G, {\bf k}) \cong K^{*+1} (  G, {\bf  k}, \RR/\ZZ), 
\]
  which  is a subgroup of differential twisted K-theory $\check{K} (G, \check{{\bf  k}})$
by  Theorems  \ref{diff:twsitedK1} and \ref{diff:twsitedK0}, under a homomorphism sending
$[f] $ to an element  $[f, \eta] \in \check{K} (G, \check{{\bf  k}})$  such that $ch_{\check{{\bf  k}} } ( f ) = -   (d-H) \eta.$  Therefore, 
the differential twisted Chern character of   $\iota_!^K (E ) $, for a complex vector bundle $E$ 
over $G/T$,  has the form 
 $(d-H)\eta (E)$ for some differential form $\eta(E)$.  Then  we have 
 \[
\int_G (d-H)\eta (E) = \int_G  (d-H)\eta (L_{k\l_i} \otimes E) \mod c_G(k)
\]
where $c_G(k)$ is the cyclic order   of   $K^{\dim G}(G, {\bf  k})$. 

 As $G/T$ is a complex manifold, the normal bundle of $G/T$ in $G$ has a canonical 
 $Spin^c$ structure,  we have  the following 
 integral  version of our Riemann-Roch Theorem \ref{RR:1}  (Cf. Remark \ref{remark:RR})
 \ba\label{index}
\int_G (d-H)\eta( E)   = \int_{G/T} Ch(E) \Td (G/T) 
\na
where $\Td (G/T)$ is the Todd class of $G/T$. Hence, we obtain 
\[
 \int_{G/T} Ch(E) \Td (G/T)  =  \int_{G/T} Ch(L_{k\l_i} \otimes  E) \Td (G/T) \mod c_G(k),
 \]
 for $i=1, \cdots,  n$.   In particular, let $E$ be  a trivial line bundle, then we have 
 the following relations
 \ba\label{dim1}
  \int_{G/T} Ch(L_{k\l_i}) \Td (G/T) = 1 \mod    c_G(k),
  \na
  for $i=1, \cdots,  n$,  where the left hand side is the dimension of the irreducible representation of $G$ with  highest weight $k\l_i$.  We can repeat the above arguments  by attaching
  3-dimensional  cells along various $S_{\a_i}^2$ simultaneously and we  obtain
  the following general relations
\ba\label{dim2}
  \int_{G/T} ch(L_{k\l_{i_1} +k\l_{i_2} +\cdots +k\l_{i_m}}) \Td (G/T) = 1 \mod    c_G(k),
  \na
  for any subset $\{i_1, i_2, \cdots, i_m\}$ of $\{1, 2, \cdots, n\}$. Denote by
  $
  V(k\l_{i_1} +k\l_{i_2} +\cdots +k\l_{i_m})
  $
  the irreducible representation of $G$ with   highest  weight $k\l_{i_1} +k\l_{i_2} +\cdots +k\l_{i_m}$, then we have the following identities relating the cyclic order of
   $K^{\dim G}(G, {\bf  k})$ to the highest weight   irreducible representations of $G$
 \ba\label{dim}
 \dim V(k\l_{i_1} +k\l_{i_2} +\cdots +k\l_{i_m}) = 1 \mod    c_G(k),
  \na   
 for any subset $\{i_1, i_2, \cdots, i_m\}$ of $\{1, 2, \cdots, n\}$.

\subsection{The  $SU(2)$  case}

In this subsection  $X= SU(2)=S^3$.  It is sufficient to consider 
an open cover $U_0,U_1$ consisting of slightly extended hemispheres with 
an intersection homotopic  to $S^2.$ In this case 
it was shown in \cite{Ros} that
$K^1(SU(2), {\bf  k} ) \cong \ZZ_k,$
and the  cyclic order agrees with the relation given by (\ref{dim}).

In the twisted cocycle of vector bundles we have now only a single element
$$(E_{01}, \nabla_{01})  \longrightarrow  U_{01}$$
 with no condition on the vector bundle. The only topological 
information are the rank $n_{01}$ of the vector bundle (which could be negative
for formal differences of vector bundles) and the total degree on $U_{01},$ which 
corresponds to an integer $m$ times the Chern class of the basic complex line 
bundle on $S^2, $ represented by the first Chern form 
$$\omega^{[2]}_{01}=c_1(E_{01}, \nabla_{01}).$$ 

 The Chern character  form of $(E_{01}, \nabla_{01})$ modulo the Chern
character of 
$$ (E_1, \nabla_1)  - (E_0, \nabla_0)  \otimes L_{01},$$
where $L_{01}$ is a complex line bundle on the intersection of degree $k,$ 
is then equal to $\mathbb Z_k$ since both $E_0$ and $E_1$ is trivial and they are 
characterized by the ranks $n_0, n_1.$ This is actually just the standard argument 
using the Mayer-Vietoris sequence for two open sets.

Choosing a partition of unity $\rho_0, \rho_1$ subordinate to the open cover $U_i$ and 
forms $B_i$ with 
$$dB_i = H= k H_0,\qquad B_0 -B_1= c_{01}$$ 
 on the overlap $U_{01}= U_0\cap U_1$,  where
 $H_0$ is the normalized volume 
form on $SU(2)= S^3$,   we obtain the 
the $(d-H)$-exact form on $SU(2)$ representing the twisted Chern differential character form
$\Theta$,
\[
\begin{array} {lll}
\Theta_0 & =&  (d-H)[ \rho_1 n_{10} + \rho_1 (\omega^{[2]}_{10} + B_0 n_{10}) ] \\[2mm]
&=& n_{10} d\rho_1 + d\rho_1 \wedge \omega^{[2]}_{10}  + n_{10} d\rho_1 \wedge B_0.\end{array}
\]
on   $U_0$, and 
\[
\begin{array} {lll}\Theta_1 &=&  (d-H)[ \rho_0 n_{01} +\rho_0 (\omega^{[2]}_{01} +B_1 n_{01} )]\\[2mm] 
&=& n_{01} d\rho_0 + d\rho_0 \wedge \omega^{[2]}_{01}  + n_{01} d\rho_0 \wedge B_1.\end{array}
\]
on   $U_1.$   Note that $n_{01} =-n_{10} $, $d\rho_0 + d\rho_1 =0$  and 
$\omega^{[2]}_{10} + B_0 n_{10} = - (\omega^{[2]}_{01} +B_1 n_{01} )$  imply that
$\Theta_0= \Theta_1$ on $U_{01}$. As $H^{odd}(X, d-H ) =0$, we know
that
\[
\Theta = (d-H) ( \eta_{[0]}  + \eta_{[2]})
\]
for a globally defined even form $( \eta_{[0]}  + \eta_{[2]})$, called a twisted eta potential. We will
show that these twisted eta potentials are localized to certain conjugacy classes  
such that 
$$
\int_{SU(2)}\eta_{[0]} H_0 = \dfrac{m}{k}  \mod 1 
$$ using the Riemann-Roch theorem in twisted K-theory (Theorem \ref{RR:1}).
 
  A Fredholm operator realization of the  twisted K-theory
classes in $ K^1(SU(2), {\bf  k} )$ is obtained from the family of hamiltonians in a supersymmetric Wess-Zumino-Witten 
model, \cite{Mi}. Actually in this case all the classes can be realized as equivariant twisted K-theory
classes, equivariant under the conjugation action of $SU(2)$ on itself. 
We recall some basic properties of the Fredholm family  from \cite{MicPel}, \cite{FHT}.

We have self-adjoint Fredholm operators $Q_A$ in a fixed Hilbert space $\cH$ parametrized 
by $SU(2)$ connections $A$ (termed `vector potentials' in the physics literature)
 on the unit circle $S^1.$ These transform equivariantly 
under gauge transformations,
\[
 \hat g^{-1} Q_A \hat g = Q_{A^g} 
 \] 
where $g$ is an element of the loop group $LG$, $\hat g$ is a projective representation of 
the loop group in $\cH$, i.e., a representation of the standard central extension of level $k,$ 
and $A^g = g^{-1} A g +g^{-1} dg$ is the gauge transformed vector potential. 

Let $p:  \mathcal{A}   \to G$ be the canonical projection from the space of vector potentials on the 
circle to the group of holonomies around the circle; the fiber of this projection is the group
of based loops $\Omega G \subset LG.$  The spectrum of $Q_A$ depends only on the 
projection $p(A) \in G,$ by the equivariance property.

For certain  $A$ the operators $Q_A$ have kernels.
Those $A$ for which this is true are such that $p(A)$ is a 
conjugacy class $\cC_j$, the so called
`D-brane', 
in $G.$  The physics terminology is to say that the zero modes 
of the family are localised in $\cC_j$.

The conjugacy class is diffeomorphic to the sphere $S^2$ given by 
$$\cC_j = \{ ghg^{-1} | g \in G\}$$
where $h = e^{i\pi \frac{2j+1}{k} \sigma_3}$ with $\sigma_3 = diag(1,-1)$ and $2j=0,1,2, \dots k-2.$ 
The zero modes form a complex line bundle $L_j$ over $\cC_j$ with Chern class represented by 
$2j+1$ times the basic 2-form on $S^2 \cong  \cC_j.$ 

For the later example it is worth noticing here that under the map $p:  \mathcal{A}   \to G$,
  the conjugacy classes correspond to
 coadjoint orbits in the Lie algebra of the central extension of the loop algebra;
the space of vector potentials on the circle is the 
(non-centrally extended) loop algebra, the coadjoint
action corresponds to the gauge action on $\mathcal A.$

If we consider a non-zero eigenvalue $\lambda$ of $Q_A$ we can still 
conclude from the continuity of the family $Q_A$ as a function of $A$ that for sufficiently small 
values of $|\lambda|$ the eigenvectors corresponding to this eigenvalue are localized at 
a 2-sphere close to the 2-sphere $\cC_j$ defined by the zero modes and the eigenvectors corresponding 
to $\lambda$ form a complex line bundle of winding number $2j+1.$ 
Thus if we fix a small real 
number $0< \epsilon$ the spectral subspace $ -\epsilon < Q_A < \epsilon$ is a complex line 
in a tubular neighborhood $\tilde \cC_j$ of $\cC_j.$ 

Setting $$U_{\pm} = \{g \in G | \pm \epsilon 
\notin \Spec(Q_A), A\in p^{-1} (g) \},$$ the intersection $U_-\cap U_+$ consists of $\tilde \cC_j$ and 
two contractible components $D_{\pm}$ (upper and lower hemispheres in $S^3$).  
In this case we have only one spectral vector bundle $E_{-+}$ which is the extension of
$L_j$ to the tubular neighborhood $\tilde \cC_j$ and a trivial line bundle on the remaining 
components $D_{\pm}$ of $U_-\cap U_+.$ 
The constraint $0 < 2j+1 < k$ is compatible with the known result $K^1(SU(2), {\bf  k} ) = \mathbb Z/k \mathbb Z.$ 
The construction gives a realization for all elements
 except the neutral element in $K^1(SU(2), {\bf  k} )$. 
The neutral element can then be obtained for example as the sum of classes corresponding to
$2j+1 = 1$ and $2j+1= k-1.$  

Write  the normalized curvature  $H$ of $\check{{\bf  k}}$ as 
$ k H_0$, where $H_0$ represents the basic 3-form on $SU(2)$. 
As the twisted cohomology $H^{odd}(SU(2), d-  kH_0)=0$, instead we  study the differential
twisted Chern character form  by applying the 
Riemann-Roch theory (Theorem \ref{RR:1}). 

By (\ref{index}), we have 
\[
\begin{array}{lll}
\disp{\int_{SU(2)} } (d-H) (\eta_{[0]} + \eta_{[2]})    & = &  
\disp{  \int_{\cC_j} }  c_1(L_j, \nabla_j ) \mod k 
\\[4mm]
&=&  (2j+1)  \mod k.
\end{array}
\]
Thus  we have, as $H=kH_0$, 
\[
\int_{SU(2)}  \eta_{[0]} H_0 =  -\dfrac{2j+1}{k}\mod 1 .
\]
 Hence, the twisted eta potential distinguishes the twisted K-classes in $K^1(SU(2), {\bf  k})$.

\subsection{The $SU(3)$ case}
As the final example we study the odd K-group $K^1(SU(3), {\bf  k}).$  
We need a representation of the twisted affine Lie algebra $ A_2^{(2)}.$ 
Here the twist refers to an outer automorphism $\tau$ of $\bold{su(3\
)}$ with $\tau^2 =1.$
Then this algebra is defined as a 
central extension of the subalgebra $L_{\tau} G$ consisting of smooth maps $g:[0,\pi] \to \bold{su(3)}$ such 
that $g(\pi) = \tau(g(0)).$ 
In the case of $SU(3)$ we can
take $\tau$ to be given by complex conjugation of matrices.  This 
automorphism of the Lie algebra integrates 
to an automorphism of the group $SU(3),$ again given by complex conjugation. 

The loop algebra of $SO(3)$ is clearly a subalgebra of $L_{\tau} G$ and the central extension of the former 
on level $k$ is obtained as a restriction of the central extension of level $k$ of the latter.  

The gauge conjugation action sending
$Q_A \to \hat g^{-1} Q_A \hat g = Q_{A^g}$ is now given by the coadjoint right action 
$A\mapsto A^g$ where the `vector potential'  $A$ is an element in the dual $L_{\tau} \bold{su(3)}^*.$ The 
coadjoint orbits have been analyzed in \cite{Wendt} and are shown to
be equal to the so called \it twisted conjugacy classes \rm 
in $SU(3).$ Recall that in this
terminology a
twisted conjugacy class $C(h)$ corresponding to $h\in SU(3)$ is defined as
$$ C(h)=  \{gh\tau(g)^{-1} | g \in SU(3) \}.$$ 

The general formula for the (twisted) conjugacy class corresponding to the zero modes of the operators $Q_A$ is 
$h= \exp{[ -2\pi (\lambda^{\vee} + \rho^{\vee})/k]}$ where $\rho$ is half the sum of positive roots of the Lie algebra of
constant gauge transformations (which in this case is $\bold{so(3)}$) and 
$\lambda$ is the highest weight of an irreducible $\bold{so(3)}$ representation, \cite{FHT}. Here $\lambda^{\vee}  \in \bold{h}$ is 
the dual of a weight $\lambda \in \bold{h}^*,$ the duality is determined by the Killing form.  Note also that the level 
$k \geq \kappa,$ where $\kappa$ is the dual Coxeter number (in the case of $SU(n)$ this number is equal to
$n$).  The reason for this is that in the Wess-Zumino-Witten model construction $k=k' + \kappa$ where $k'$ is the level
of an arbitrary irreducible loop group representation and the shift $\kappa$ comes from a loop group representation
constructed from the Clifford algebra of the loop group. 

As shown in \cite{Stanciu} the twisted conjugacy classes $C_x$ defined by $h(x) = \exp{2x \rho^{\vee}}$  define a foliation of $SU(3)$ when
$0\leq x \leq \pi/4.$ Concretely, as a $3\times 3$  matrix,
$$ h(x) = \left( \begin{matrix} \cos(2x) & \sin(2x)  & 0 \\ -\sin(2x) & \cos(2x) & 0 \\ 
0 & 0 & 1 \end{matrix} \right). $$ 
When $x=0$ the twisted conjugacy class can be identified as the 5-dimensional space $M_5=SU(3)/SO(3), $
when $x=\pi/4$ it is $S^5 = SU(3)/SU(2)$ whereas for $0<x< \pi/4$ we obtain the 'seven brane' $M_7=SU(3)/ SO( 2).$ 

The class defined by the zero modes of $Q_A$ is the generic orbit $M_7$ corresponding to the parameter value
$x = (2j + 1)/k,$ where $2j=0,1,2 \dots k-3,$  the shift by 3 coming from the dual Coxeter number of $SU(3).$ 
In the case of the twisted loop algebra $A^{(2)}_2$ we have an additional constraint: The spin $j$ is integer when 
$k-3$ is even and $j$ is a half-integer when $k-3$ is odd. 

The degree of the zero mode line bundle is computed as in the $SU(2)$ case. The zero modes are localized 
in the finite-dimensional vacuum subspace of the  Hilbert space $H_{j,k}$ which carries a representation $j\otimes 1/2$ 
of $\bold{so(3)}.$ (In general, the zero mode bundle has rank equal to the multiplicity of the representation $\rho$ inside of the spin 
representation of the Clifford algebra of  the group of constant loops; here however the Clifford algebra representation
is two dimensional and at the same time the irreducible fundamental representation of the Lie algebra $\bold{so(3)}.$)
 When $x= (2j+1)/k$ the complex line is spanned by the highest weight vector $v_{j+1/2}$ of $\bold{so(3)}$ 
  weight $j+1/2,$ \cite{FHT}. The zero mode bundle is then the associated line bundle $L= SU(3) \times_{2j+1} \mathbb C,$ defined by
the principal bundle $SO(2) \to SU(3) \to M_7$ through the one dimensional representation of $SO(2)$ with character 
$2j +1.$  In particular, for even $k$ the spin $j$ is a half-integer and thus the degree of the zero mode bundle is 
even. This is in accordance with the known result $K^1(SU(3), {\bf  k}) = \mathbb Z/(k/2)\mathbb Z$ for even $k.$ Twisted $K^1$ 
is a rank one module over the untwisted $K^0(SU(3))$, tensoring with any of the elements in $K^1(SU(3), {\bf  k})$ gives 
only even elements in $\mathbb Z_k.$ 

In the case when $k$ is odd, the degree of the zero mode bundle is also odd and the tensor product operation 
$K^0(SU(3)) \times K^1(SU(3), {\bf  k}) \to K^1(SU(3),{\bf  k} ),$ gives both the even and odd elements in $\mathbb Z_k.$ 

The eta forms detect the twisted K-theory classes in a similar way as in the case of $SU(2).$  In the present 
setting we have nonzero eta forms in even degrees up to form degree six. However, they do not contain 
independent information since for any fixed level $k$ the only parameter is the twisting of the line bundle 
$L$ which is given by the integer $2j+1.$ To determine this integer from the differential data one proceeds
as in the case of  $SU(2)$ above.  Consider the subgroup $SU(2) \subset SU(3)$ given by the the matrices  
with $+1$ on the diagonal in the lower right corner. The intersection of $SU(2)=S^3$ with the orbit $M_7$ is a
union of two spheres
 $S^2_a$ and $S^2_b.$ The reason for this is that the points $h(x)$ and $h(-x)$ are conjugate to each other
 by the twisted action of $g=diag(1,-1,-1)$ which lies outside of $SU(2);$ one can see easily that $h(\pm x)$ are not conjugate
 by elements of $SU(2)$ and their union is $M_7 \cap SU(2).$ So the differential twisted Chern 
character form  is given by 
 \[
 (d-H) ( \eta_{[0]} + \dots +  \eta_{[6]}).
\]
  Picking up the component of form degree 3 we get
\[
\begin{array}{lll}
\disp{\int_{SU(2)} } \eta_{[0]} H  & = & \int_{S^2}   c_1(L_j, \nabla_j)   \mod c_G(k) \\[2mm]
&=&  2(2j+1)    \mod c_G(k) , 
\end{array}
\]
the factor $2$ coming from the integration of $c_1$ over the two distinct 2-spheres. Thus again,   we get 

\[
\int_{SU(2)}  \eta_{[0]} H_0 = -  \dfrac {2(2j+1)}{k} \mod c_G(k)/k.
\]

  \end{document}